\documentclass[11pt,reqno]{amsart}
\usepackage{amsmath,amsfonts,amsthm,amssymb,latexsym,amscd,amsopn,eucal,graphics,color,enumitem,lscape}
\usepackage[all,cmtip]{xy}
\usepackage[margin=1in]{geometry}
\usepackage[dvipsnames]{xcolor}
\usepackage[colorlinks]{hyperref}
 \hypersetup{
  colorlinks,
  citecolor=Mulberry,
  linkcolor=Red,
  urlcolor=Green}
\usepackage{comment}
\usepackage{mathrsfs}
\usepackage[T1]{fontenc}

\DeclareFontEncoding{OT2}{}{} 

\def\Sel{{\rm Sel}}

\def\SL{{\rm SL}}
\def\GL{{\rm GL}}
\def\PGL{{\rm PGL}}

\def\Cl{{\rm Cl}}
\def\O{{\mathcal O}}
\def\SO{{\rm SO}}
\def\P{{\mathbb P}}

\def\irr{{\rm irr}}

\def\Vol{{\rm Vol}}
\def\R{{\mathbb R}}

\def\FF{{\mathscr{F}}}
\def\GG{{\mathscr{G}}}

\def\Q{{\mathbb Q}}

\def\Z{{\mathbb Z}}
\def\P{{\mathbb P}}

\def\FF{{\mathscr F}}
\def\Q{{\mathbb Q}}

\def\cE{\mathcal{E}}
\def\p{\mathfrak{p}}

\newcommand{\FunivX}{\FF_{\mathrm{univ}}^{\leq X}}
\newcommand{\FFT}{\FF^{\leq T}}
\newcommand{\Funiv}{\FF_{\mathrm{univ}}}

\newcommand{\floor}[1]{\left\lfloor #1\right\rfloor}
\newcommand{\Avg}{\mathrm{Avg}}
\newcommand{\eps}{\varepsilon}
\newcommand{\rank}{\mathop{\mathrm{rank}}}

\renewcommand{\Cl}{\mathrm{Cl}}

\newcommand{\Nm}{\mathrm{Nm}}
\newcommand{\Norm}{\mathrm{Norm}}
\newcommand{\Frac}{\mathrm{Frac}\,}
\newcommand{\nfrak}{\mathfrak{n}}
\newcommand{\pfrak}{\mathfrak{p}}
\newcommand{\dfrak}{\mathfrak{d}}
\newcommand{\afrak}{\mathfrak{a}}
\newcommand{\Qbar}{\overline{\mathbb{Q}}}

\newcommand{\Prob}{\mathrm{Prob}}
\newcommand{\zfrak}{\mathfrak{z}}
\newcommand{\moment}{r}

\theoremstyle{plain}
  \newtheorem{theorem}{Theorem}[section]
  \newtheorem{proposition}[theorem]{Proposition}
  \newtheorem{corollary}[theorem]{Corollary}
  \newtheorem{lemma}[theorem]{Lemma}
  
\theoremstyle{definition}
  \newtheorem{definition}[theorem]{Definition}

  \newtheorem{remark}[theorem]{Remark}
  
\theoremstyle{remark}

\usepackage{etoolbox}

\makeatletter
\patchcmd{\@settitle}{\uppercasenonmath\@title}{\Large}{}{}
\patchcmd{\@setauthors}{\MakeUppercase{\authors}}{{\large \authors}}{}{}
\makeatother

\makeatletter
\let\@@pmod\pmod
\DeclareRobustCommand{\pmod}{\@ifstar\@pmods\@@pmod}
\def\@pmods#1{\mkern4mu({\operator@font mod}\mkern 6mu#1)}
\makeatother

\title[{The second moment of the number of integral points on elliptic curves is bounded}]{The second moment of the number of integral points \\  on elliptic curves is bounded}
\author{Levent Alp\"{o}ge}
\author{Wei Ho}

\begin{document}

\begin{abstract}
Let $K$ be a number field and $S$ a finite set of places of $K$ containing all archimedean places. In this paper, we show that the second moment of the number of $S$-integral points on elliptic curves over $K$ is bounded. In particular, we prove that, for any positive real number $\moment\leq \log_2 5 = 2.3219 \ldots$, the $\moment$-th moment of the number of $S$-integral points is bounded for the family of all integral short Weierstrass curves ordered by height, or for any positive density subfamily thereof. For certain other families of elliptic curves over $\Q$, such as those with one or two marked points, we prove that the average of the number of integral points is bounded; in fact, for the family with one marked point, the $\moment$-th moment is also bounded for all positive $\moment \leq \log_2 3$.

The essential new ingredient in our proof is an upper bound on the number of $S$-integral points on an affine integral Weierstrass model $\cE$ of an elliptic curve $E$ over $K$ depending on the rank of the curve, the class group and degree of $K$, and the number of primes of $K$ whose square divides the discriminant of $\cE$. For example, over $\Q$, the bound for integral points on $\cE$ is $2^{\rank{E(\Q)}} O(1)^s$, where $s$ is the number of prime squares dividing the discriminant of $\cE$.

The theorems on moments then follow from averaging this new upper bound; crucially, we can bound the average of the $2^{\rank}$ term by using results on the average sizes of Selmer groups in the families. In order to prove the bounds for the $\moment$-th moment when $\moment = \log_2 5$ (and the analogous equality cases for the other families), we introduce a method to count orbits of coregular representations with possibly unbounded weights.
\end{abstract}

\maketitle

\section{Introduction}

In this paper, we prove several theorems about the number of integral points on elliptic curves over a number field $K$. We bound the number of integral points using the rank of the elliptic curve, the higher order divisors of its discriminant, and the degree and class number of $K$. For any finite set $S$ of places of $K$ containing all archimedean places, we also obtain an analogous bound for $S$-integral points. Using this bound, which is stronger on average than previous such bounds, we prove that the second moment of the number of ($S$-)integral points on elliptic curves over $K$ is bounded. 

We first give an explicit upper bound on the number of integral points on affine integral Weierstrass models of elliptic curves over $\Q$, depending only on the rank and the number of square divisors of the discriminant of the curve:

\begin{theorem}\label{thm:the bound on the number of integral points}
Let $A,B\in \Z$ be such that $\Delta_{A,B} := -16(4A^3 + 27B^2)\neq 0$. Let $\mathcal{E}_{A,B}$ be the affine integral model $y^2 = x^3 + Ax + B$ of the associated elliptic curve $E_{A,B}$ over $\Q$. Then
$$\left| \mathcal{E}_{A,B}(\Z) \right|  \ll 2^{\rank{E_{A,B}(\Q)}} \prod_{p^2 \mid \Delta_{A,B}} \min\left(4 \floor{\frac{v_p(\Delta_{A,B})}{2}} + 1, 7^{2^7} \right).$$
\end{theorem}

Here $v_p$ denotes the $p$-adic valuation for a prime $p$, and $\mathcal{E}_{A,B}(R)$ denotes the set of solutions $\{(x,y)\in R^2 : y^2 = x^3 + Ax + B\}$ for any ring $R$. By $f\ll g$ we mean that there is a positive absolute constant $c > 0$ such that $|f|\leq c |g|$.

\vspace{\baselineskip}

Our general bound for $S$-integral points for an elliptic curve over a number field $K$ is as follows:

\begin{theorem} \label{thm:optimalgeneralbound}
Fix $C = 7^{2^7}$. Let $K$ be a number field, and let $\O_K$ denote its ring of integers. Let $A,B \in \O_K$ such that $\Delta_{A,B} := -16(4A^3 + 27B^2)\neq 0$. Let $S$ be a finite set of places of $K$ containing all infinite places and all primes $\pfrak$ for which $\pfrak^2 \mid \Delta_{A,B}$, and let $\O_{K,S}$ denote the ring of $S$-integers in $K$, and let $\Cl(R)$ denote the class group of the ring $R$.

Let $\mathcal{E}_{A,B}: y^2 = x^3 + Ax + B$ be an affine Weierstrass model of the elliptic curve $E_{A,B}$ over $K$. Then we have the bound
\begin{equation} \label{eq:generalboundSint}
|\mathcal{E}_{A,B}(\O_{K,S})|\leq 2^{\rank E_{A,B}(K)} C^{2|S|+1} |\Cl(\O_{K,S})[2]|.
\end{equation}
\end{theorem}

\begin{remark} There are several weaker but simpler variants of the bound \eqref{eq:generalboundSint}. First, since $\Cl(\O_{K,S})$ is a quotient of $\Cl(\O_K)$, one may replace $\Cl(\O_{K,S})[2]$ with $\Cl(\O_K)[2]$, or even just $\Cl(\O_K)$.
Moreover, taking $S$ to be as small as possible, i.e., the union of the infinite places and the primes $\p$ with $v_{\p}(\Delta_{A,B})\geq 2$, we obtain the following bound on integral points:
\begin{equation*}
\left| \mathcal{E}_{A,B}(\O_K) \right| \leq 2^{\rank E_{A,B}(K)} C^{2[K:\Q] + 2\omega_{\geq 2}(\Delta_{A,B})+1} \left| \Cl(\O_{K})[2] \right|,
\end{equation*}
where $\omega_{\geq 2}(\Delta_{A,B}) := \#\{\text{primes } \p : v_{\p}(\Delta_{A,B})\geq 2\}$.
\end{remark}

\begin{remark}
The right-hand side of \eqref{eq:generalboundSint} must depend on $S$: if $\rank{E_{A,B}(K)} > 0$, the left-hand side may be made arbitrarily large by expanding $S$.
\end{remark}

\begin{remark}
Although $\cE_{A,B}$ and $\cE_{d^4 A,d^6 B}$ for $0\neq d \in \O_K$ give isomorphic elliptic curves, the latter may have many more integral points. This example shows that there is no uniform bound on the number of $\O_{K,S}$-points for all $\cE_{A,B}$ as $A$ and $B$ vary freely, though it does not rule out such a bound for, say, quasi-minimal short Weierstrass models, i.e., those such that $\mathrm{Norm}_{K/\Q} (\Delta_{A,B})$ is minimized subject to $A, B \in \O_K$.
\end{remark}

Mordell was the first to prove the finiteness of the number of integral points on an elliptic curve (essentially by the invariant-theoretic method we employ in this paper), a theorem generalized by Siegel to all curves of genus $g\geq 1$. Since then, there has been significant work on bounding the number of integral points on elliptic curves.
For example, Helfgott and Venkatesh \cite{helfgottvenkatesh} prove that, for any integral short Weierstrass model $\mathcal{E}_{A,B}$ of the curve $E_{A,B}$ over $\Q$,
\begin{equation} \label{eq:HelfgottVenkateshBound}
\left| \mathcal{E}_{A,B}(\Z) \right|  \ll 1.33^{\rank{E_{A,B}(\Q)}} O(1)^{\omega(\Delta_{A,B})} (\log \left| \Delta_{A,B} \right|)^2
\end{equation}
where $\omega(n)$ denotes the number of distinct prime factors of $n$; they also
give a bound
depending only on a small power of the discriminant, which was recently improved by Bhargava, Shankar, Taniguchi, Thorne, Tsimerman, and Zhao \cite{BSTTTZ} :
\begin{equation} \label{eq:BSTTTZ}
\left| \mathcal{E}_{A,B}(\Z) \right| \ll_\eps \left|\Delta_{A,B}\right|^{0.1117\ldots+\eps}.
\end{equation}
For quasi-minimal short Weierstrass models $\mathcal{E}_{A,B}$ of elliptic curves $E_{A,B}$ over $K$, Silverman \cite{silverman-quantSiegel} shows that
$$\left| \mathcal{E}_{A,B}(\O_{K,S}) \right| \ll O_K(1)^{(1+\rank{E_{A,B}(K)})(1+\omega_{\text{ss}}(\Delta_{A,B}))+|S|}$$
where $\omega_{\text{ss}}(\Delta_{A,B})$ denotes the number of primes of semistable bad reduction and the $O(1)$ is already on the order of $10^{10}$ for $K = \Q$.
In fact, stronger bounds
\begin{equation} \label{eq:hindrysilverman}
\left| \mathcal{E}_{A,B}(\O_{K,S}) \right| \ll O_K(1)^{(1+\rank{E_{A,B}(K)})(1 + \sigma_{A,B}) + |S|},
\end{equation}
and for curves with $h(\Delta_{A,B})\geq \frac{1}{2}h(j(E_{A,B}))$, by David \cite{sinnoudavid}:
\begin{equation*}
\left| \mathcal{E}_{A,B}(\O_{K,S}) \right| \ll O_K(\sigma_{A,B}^6 (\log{(1 + \sigma_{A,B})})^3)^{1+\rank{E_{A,B}(K)} + |S|}.
\end{equation*}
Here $\sigma_{A,B} := \frac{\log \mathrm{Norm}_{K/\Q} (\Delta_{A,B})}{\log \mathrm{Norm}_{K/\Q} (N_{A,B})}$ is the Szpiro ratio of $E_{A,B}$ over $K$ and $N_{A,B}$ denotes the conductor of $E_{A,B}$.

\begin{remark} \label{rmk:strongeronaverage}
These previous bounds are not suitable for our applications on moments, which we obtain by averaging the bound \eqref{eq:generalboundSint} from Theorem \ref{thm:optimalgeneralbound}.
In particular, we control the first factor in \eqref{eq:generalboundSint} (namely, $2^{\rank}$) using results on the average sizes of $d$-Selmer groups for small $d$, but the implicit constants in, e.g., \eqref{eq:hindrysilverman} are far too large for those average Selmer results to apply. Moreover, the second factor in \eqref{eq:generalboundSint} (namely, $O(1)^{|S|}$) is controlled on average because the size of $S$ depends only on the prime {\em squares} dividing the discriminant, whereas even if \eqref{eq:BSTTTZ} or \eqref{eq:HelfgottVenkateshBound} were generalized to other number fields $K$, both $|\Delta_{A,B}|^{0.1117\ldots+\eps}$ and the $O(1)^{\omega(\Delta_{A,B})} (\log \left| \Delta_{A,B} \right|)^2$ factor diverge on average.
\end{remark}

\begin{remark}
Since the ABC conjecture implies that the Szpiro ratio is at most $6 + o(1)$, the Hindry--Silverman bound \eqref{eq:hindrysilverman} implies that, conditional on ABC and uniform boundedness of ranks for elliptic curves over $K$, the number of $S$-integral points is uniformly bounded for quasi-minimal elliptic curves over a fixed $K$. (In fact, all one needs is Lang's conjecture that $\widehat{h}(P)\gg h(E_{A,B})$ for non-torsion $P\in E_{A,B}(K)$.) Alternatively, Abramovich \cite{abramovich} has shown that the Lang--Vojta conjecture for varieties of log general type implies uniform boundedness of the number of $S$-integral points on so-called stably minimal models of elliptic curves.
\end{remark}

\vspace{\baselineskip}

We use Theorem \ref{thm:optimalgeneralbound} to prove that the second moment of the number of $S$-integral points on elliptic curves over $K$ is bounded.
We consider the family $\Funiv(\O_K)$ of all integral short Weierstrass models
$$\mathcal{E}_{A,B} : y^2 = x^3 + A x + B$$
of elliptic curves over $K$, where $A, B \in \O_K$ with $\Delta_{A,B} \neq 0$, and order this family
by the {\em height}
\begin{equation} \label{eq:ourheightdefinition}
    H(\cE_{A,B}) = H(A,B) := \prod_{v\mid \infty} \max(|A|_v^{1/4}, |B|_v^{1/6}).
\end{equation}
The same boundedness-of-moments results also hold when ordering elliptic curves over $K$ instead by the usual height on the weighted projective line $\P(4,6)$, which gives the height $\widetilde{H}(\cE_{A,B}) := \Norm(I(A,B)) \prod_{v \mid \infty}\max(|A|_v^{1/4}, |B|_v^{1/6})$ where $I(A,B) := \{a\in K : a^4 A, a^6 B\in \O_K\}\subseteq K$. See Remark \ref{rmk:otherheights}.

Not only do we prove that the second moment of the number of integral points in this family is bounded, we obtain the following slightly stronger result. 

\begin{theorem}\label{thm:the bound on the moments}
Fix $C = 7^{2^7}$. Let $K$ be a number field, and let $S$ be a finite set of places of $K$ containing all infinite places.
Let $\FF$ be a subset of $\Funiv(\O_K)$ of positive lower density (when ordering by height), and let $\moment\leq \log_2{5} = 2.3219\dotso$ be a positive real number. We have
\begin{equation} \label{eq:momentbounded}
\Avg_{\FF}(\left|\cE_{A,B}(\O_{K,S})\right|^\moment)\ll_{\FF} \left( C^{2|S|} \left| \Cl(\O_{K,S})[2] \right| \right)^\moment
\end{equation}
where the average is taken over all $\cE_{A,B}\in \FF$ ordered by height.
\end{theorem}

More precisely, let
$$\FF^{\leq T} := \{(A,B)\in \FF : \Delta_{A,B} \neq 0, H(A,B) \leq T \}$$
be the set of all $(A,B)\in \FF$ with height up to $T$.
Then there exists a constant $c_{\FF}$,\footnote{One may take $c_{\FF}$ to be an absolute constant divided by the lower density of the family $\FF$ in $\Funiv(\O_K)$, or even 
$\ll \limsup_{T\to \infty} {|\mathcal{L}^{\leq T}|}/{|\FF^{\leq T}|}$ where $\FF\subseteq \mathcal{L}$ and $\mathcal{L}$ is a ``large family'' in the sense of \cite{BSW-globalfields2}.
}
depending only on $\FF$, such that
$$\limsup_{T\to\infty} \frac{\sum\limits_{(A,B)\in \FF^{\leq T}} \left|\cE_{A,B}(\O_{K,S})\right|^\moment}{\left|\FF^{\leq T}\right|}\leq c_{\FF} \left(C^{2|S|} \left|\Cl(\O_{K,S})[2]\right|\right)^\moment.$$

\begin{remark}
More generally, using Theorem \ref{thm:optimalgeneralbound}, we have that the same moments of $S$-integral points on elliptic curves are bounded even when the set $S$ is allowed to vary with $(A,B)$, as long as the number of primes in the set $S(A,B)$ for each $(A,B)$ does not grow too quickly (indeed, so long as $\Avg\,{t^{|S(A,B)|}} < \infty$ for a sufficiently large constant $t\ll_\moment 1$).
\end{remark}

One expects that elliptic curves should have no ``unexpected points'' on average, i.e., that all these moments should be $0$ (note that the point at infinity is not an integral point).
In \cite{levent-intpts}, it is proved that for $0 < \moment < \log_3 5 = 1.4649\ldots$, the average in \eqref{eq:momentbounded} for integral points on elliptic curves over $K = \Q$ is bounded (and thus by taking $\moment = 1$, that the average number of integral points is bounded, a result also proved by D.~Kim \cite{dohyeoungkim-intpts}).

\begin{remark}
A related but different question is to show that most elliptic curves have very few integral points; perhaps the strongest known result in this direction is that $80\%$ of curves in $\Funiv$ have at most $2$ integral points (by combining the fact that $100\%$ of rank $1$ curves in $\Funiv$ have at most $2$ points \cite[Lemma 20]{levent-intpts} with Bhargava--Shankar's result \cite{arulmanjul-5Sel} that at least $80\%$ of curves in $\Funiv$ have rank $0$ or $1$). Note that these bounds do not imply that the average number of integral points is bounded, since it is a priori possible that there is some small exceptional subset in which the curves have an enormous number of points.
\end{remark}

\begin{remark}
Theorem \ref{thm:the bound on the moments} gives a bound on $\Avg(\left|\cE_{A,B}(\O_{K,S})\right|^\moment)$ for $\moment\leq \log_2{5}$ by using Bhargava--Shankar's bound on the average size of $5$-Selmer groups over this family \cite{arulmanjul-5Sel}. If a bound on the average size of $d$-Selmer groups over this family were known, the same argument would yield a similar bound for all $\moment\leq \log_2 d$.
\end{remark}

\begin{remark} \label{rmk:equalitycase}
The equality case $\moment = \log_2 5$ of Theorem \ref{thm:the bound on the moments} is treated by modifying the proof of Bhargava--Shankar's average $5$-Selmer bound to instead use {\em weighted} $5$-Selmer elements (potentially with unbounded weights); see Section \ref{sec:weights}. To the best of our knowledge, this is the first time that these geometry-of-numbers techniques in arithmetic statistics have been used to count these types of (possibly heavily) weighted orbits of a representation, and we believe that this generalization may be useful in other applications.
\end{remark}

\vspace{\baselineskip} 

For two other families of elliptic curves over $\Q$, namely, the families
\begin{align*}
\FF_1 &:= \{y^2 + d_3 y = x^3 + d_2 x^2 + d_4 x \mid d_2, d_3, d_4 \in \Z, \, \Delta \neq 0 \} \\
\text{and} \qquad \FF_2 &:= \{y^2+ d_1 x y + d_3 y = (x-d_2)(x-d_2')(x-d_2'') \mid \\
& \qquad \qquad  d_1, d_2, d_2', d_2'', d_3 \in \Z,\, d_2 + d_2' + d_2'' = 0,\, \Delta \neq 0 \},
\end{align*}
of elliptic curves in Weierstrass form with one and two marked points, respectively, ordered by an analogous notion of height, we also find that the average of the number of integral points is bounded. 
\begin{theorem} \label{thm:moments for families intro}
Let $\FF$ be a subfamily of elliptic curves in $\FF_1$ (respectively, $\FF_2$) of positive lower density. For any positive real number $\moment \leq \log_2 3 = 1.5850\ldots$ (resp., $\moment \leq 1$), we have
$$\Avg_{\cE \in \FF}( \left| \cE(\Z) \right|^\moment ) \ll_\FF 1$$
where the average is taken over all $\cE\in \FF$ ordered by height.
\end{theorem}

\vspace{\baselineskip}

\subsection*{Method of proof}

Theorem \ref{thm:the bound on the number of integral points} follows from studying a bijection first observed by Mordell between integral points on an integral Weierstrass model $\mathcal{E}_{A,B}$ of an elliptic curve and binary quartics of the form $X^4 + 6 c X^2 Y^2 + 8 d X Y^3 + e Y^4$ with $c, d, e \in \Z$ and invariants $I = -48A$ and $J = -1728B$. The natural map taking the integral point to an element of the $2$-Selmer group of the elliptic curve translates precisely to taking the corresponding binary quartic to its $\PGL_2(\Q)$-equivalence class. By working explicitly (and using results of Bombieri--Schmidt \cite{bombierischmidt} and Evertse \cite{evertse} on Thue equations), we bound the size of a fiber of this map by
$$\ll \prod_{p^2\mid \Delta_{A,B}} \min\left\{ 4 \floor{\frac{v_p(\Delta_{A,B})}{2}} + 1,\, 7^{2^7} \right\}.$$
The image lies in $E_{A,B}(\Q)/2E_{A,B}(\Q)$, whose size is at most $ 4 \cdot 2^{\rank{E_{A,B}(\Q)}}$, giving the theorem. Theorem \ref{thm:optimalgeneralbound} follows the same basic strategy, after generalizing the algebraic constructions to $S$-integral points over $K$ and using the fact that Evertse in fact treats solutions to Thue-Mahler equations over $\O_{K,S}$.
 
We deduce Theorem \ref{thm:the bound on the moments} for $\moment < \log_2 5$ from Theorem \ref{thm:optimalgeneralbound}, H{\"o}lder's inequality, standard analytic techniques, and knowledge of bounds on the average sizes of $5$-Selmer groups in this family (these bounds over $\Q$ come from Bhargava--Shankar's work \cite{arulmanjul-5Sel} and have been extended to number fields $K$ by Bhargava--Shankar--Wang \cite{BSW-globalfields2}). As mentioned in Remark \ref{rmk:equalitycase}, the equality case $\moment = \log_2 5$ is treated by modifying the geometry-of-numbers proof for the average $5$-Selmer bound to count $5$-Selmer elements with large weights.

For the families $\FF_1$ and $\FF_2$, the results on moments, for $\moment \leq \log_2 3$ or $\leq 1$, respectively, follow from the same techniques as for Theorem \ref{thm:the bound on the moments} and bounds on the average $3$-Selmer (resp., $2$-Selmer) group size from \cite{cofreecounting}. Note that the equality case is of most interest for the family $\FF_2$, where it gives an upper bound for the average number of integral points on curves in $\FF_2$. We expect that identical results on bounding moments of integral points for these types of families over any number field hold, by using \cite{BSW-globalfields2} to generalize the average Selmer theorems in \cite{cofreecounting} to number fields.

\begin{remark} While there are other families of elliptic curves with known average $d$-Selmer upper bounds, in many cases (such as the family of Mordell curves and families with marked $2$- or $3$-torsion points), the discriminant often---or even always!---has square factors. We thus cannot directly use these same averaging techniques to prove analogous moment bounds for these families.
\end{remark}

\vspace{0.5\baselineskip}

\subsection*{Acknowledgments} We thank Shabnam Akhtari, Noah Albrecht, Pieter Belmans, Manjul Bhargava, Zev Klagsbrun, Erik Metz, Arul Shankar, and Joe Silverman for helpful comments and conversations. Several questions by Joe Silverman inspired us to add more general results over arbitrary number fields and for $S$-integral points. LA was supported by NSF grant DMS-2002109 and the Society of Fellows, and WH was supported by NSF grants DMS-1701437 and DMS-1844763, the Sloan Foundation, and the Minerva Research Foundation.

\subsection*{Notation}
For the remainder of the paper, let $K$ be a number field and let $\O_K$ be its ring of integers. For a finite set of places $S$ of $K$, let $\O_{K,S}$ be the ring of $S$-integers. An element of $K$ is called {\em integral} if it is in $\O_K$, and {\em $S$-integral} if it is in $\O_{K,S}$.  Any ideal denoted $\pfrak$ is prime. Affine Weierstrass models over $\O_K$ of elliptic curves are denoted $\cE$ or $\cE_{A,B}$, and the associated elliptic curve over $K$ is denoted $E$ or $E_{A,B}$. We will sometimes, by a small abuse of notation, refer to $\cE$ itself as an elliptic curve.

\section{Binary quartic forms and integral points on elliptic curves}

\subsection{Preliminaries on binary quartic forms}
Given a binary quartic form 
	\begin{equation} \label{eq:bq}
	f(X,Y) = a X^4 + b X^3 Y + c X^2 Y^2 + d X Y^3 + e Y^4
	\end{equation}
with coefficients in $K$, the group $\SL_2(K)$ naturally acts by linear substitutions of the variables, i.e., for $g \in \SL_2(K)$, one has
	\begin{equation} \label{eq:SL2action}
	g \cdot f(X,Y) = f((X,Y) \cdot g).
	\end{equation}
There exist degree $2$ and $3$ polynomial invariants $I$ and $J$ that generate the $\SL_2(K)$-invariant ring
as a polynomial ring. The standard normalizations of $I$ and $J$ are as follows:
	\begin{align*} 
	I &= 12 a e - 3 b d + c^2, \\
	J &= 72 a c e - 27 a d^2 - 27 b^2 e + 9 b c d - 2 c^3.
	\end{align*}
The discriminant $\Delta(f) = \frac{1}{27}(4I^3-J^2)$ of $f$ is a polynomial invariant with $\Z$-coefficients. It is well known that if $\Delta(f)$ is nonzero, then the double cover $Z^2 = f(X,Y)$ of $\P^1$ is a genus one curve with Jacobian isomorphic to the elliptic curve given by
	$$y^2 = x^3 - \frac{I}{3} x - \frac{J}{27}.$$
Conversely, a smooth genus one curve over $K$ with a rational degree $2$ divisor or line bundle (thereby giving a degree $2$ map to $\P^1$) has a model of the form $Z^2 = f(X,Y)$ for a binary quartic form $f$ over $K$ with $\Delta(f) \neq 0$.

Let $S$ be a finite set of primes of $K$. We say that a binary quartic form \eqref{eq:bq} is {\em $S$-integral} if $a, b, c, d, e \in \O_{K,S}$ and {\em $S$-integer-matrix} if additionally $4$ divides $b$ and $d$ and $6$ divides $c$. Both conditions are preserved by the action of $\SL_2(\O_{K,S})$. For an $S$-integer-matrix binary quartic form $f$, there are polynomial invariants $I'(f), J'(f)$ with $\Z$-coefficients such that $12 I' = I$ and $432 J' = J$, so that the elliptic curve associated to $f$ is isomorphic to the curve given by
	\begin{equation} \label{eq:ECwithIJprime}
	y^2 = x^3 - 4I' x - 16 J'.
	\end{equation}

In the sequel, we will mostly work with binary quartics of a special type, so we name them as follows:
\begin{definition}
We say a binary quartic form \eqref{eq:bq} is {\em demonic} (for ``depressed monic'') if it is monic with no $X^3 Y$-coefficient, i.e., if $a = 1$, $b = 0$, and $c, d, e \in K$.
\end{definition}

\subsection{Mordell's construction} \label{sec:Mordell}
In \cite[Chapter 25]{mordelldiophantine}, Mordell shows that, given an integral point on an affine Weierstrass model of an elliptic curve $y^2 = x^3 + Ax + B$ with $A,B \in \Z$, there exists an integer-matrix binary quartic form $f(X,Y)$ and $p,q \in \Z$ such that $f(p,q) = 1$ and $I'(f) = -4A$ and $J'(f) = -4B$; however, his construction is not explicit. Conversely, given an integer-matrix binary quartic form $f(X,Y)$ such that $I'$ and $J'$ are multiples of $4$ and $p,q \in \Z$ such that $f(p,q) = 1$, one may explicitly produce (using covariants of $f$) an integral point on the elliptic curve \eqref{eq:ECwithIJprime}.
In the next two subsections, we give a geometric explanation of Mordell's construction (generalized to $S$-integral points on elliptic curves over number fields $K$), which yields an explicit construction of a monic $S$-integer-matrix binary quartic form associated to an $S$-integral point on an elliptic curve.

Let $S$ be a finite set of places of $K$. Let $E$ be an elliptic curve over $K$ with affine $S$-integral Weierstrass model
	\begin{equation} \label{eq:Weierstrass}
	\cE_{A,B}: y^2 = x^3 + A x + B 
	\end{equation}
with $A, B \in \O_{K,S}$. Let $O$ denote the point at infinity. Given a point $P = (x_0,y_0) \in \cE_{A,B}(K)$, the degree $2$ divisor $O + P$ induces a map from $E$ to $\P^1$ as a double cover ramified at a degree four subscheme of $\P^1$. In other words, we obtain a binary quartic form over $K$, which is easily computed \cite{cremonafisherstoll, cofreecounting}:
	\begin{equation} \label{eq:BQfromintpoint}
	f(X,Y) = X^4 - 6 x_0 X^2 Y^2 + 8 y_0 X Y^3 + (-4A-3x_0^2) Y^4.
	\end{equation}
It is easy to check that $I'(f) = -4A$ and $J'(f) = -4B$.

Conversely, given a (demonic) binary quartic of the form $f(X,Y) = X^4 + 6c X^2 Y^2 + 4d X Y^3 + e Y^4$ with $a, b, c, d, e \in \O_{K,S}$, we may easily solve for the coefficients of the elliptic curve and the integral point by equating the coefficients with \eqref{eq:BQfromintpoint}. We obtain the elliptic curve 
	$$E: y^2 = x^3 - \frac{3 c^2 + e}{4} x + \frac{c^3 + d^2 - c e}{4}$$
which contains the point $P = (x_0,y_0) = (-c, d/2)$. Assuming that $I'(f) = 3c^2+e$ and $J'(f) = - c^3 - d^2 + c e$ are both divisible by $4$, we immediately have that $d$ must be divisible by $2$, in which case the elliptic curve $E$ and the point $P$ both have $S$-integral coefficients.

It is clear that these constructions are inverse to one another. We thus obtain the explicit maps for the bijection in the following theorem, a generalization of Mordell's construction: 

\begin{theorem} \label{thm:Mordell}
The following two sets are in bijection:
\begin{enumerate}[label={\upshape(\roman*)}]
\item \label{item:Weierpluspoint} $S$-integral affine Weierstrass models $y^2 = x^3 + Ax + B$ of elliptic curves with $S$-integral points $(x_0,y_0)$ (namely, with $A, B, x_0, y_0 \in \O_{K,S}$),
\item \label{item:demonicBQs} binary quartics $X^4 + 6 c X^2 Y^2 + 8 d X Y^3 + e Y^4$ with $c,d,e \in \O_{K,S}$ and $c^2 - e$ divisible by $4$.
\end{enumerate}
\end{theorem}

Note that the binary quartics in Theorem \ref{thm:Mordell} are $S$-integer-matrix and demonic. In the proof of Theorem \ref{thm:optimalgeneralbound}, we will only need that the set of {\em integral} Weierstrass models with $S$-integral points, which is a subset of the set \ref{item:Weierpluspoint}, injects into the set \ref{item:demonicBQs} of these demonic $S$-integer-matrix binary quartics.

\subsection{Binary quartics with representations of $1$}

We now relate the sets in Theorem \ref{thm:Mordell} with binary quartic forms with representations of $1$, which is Mordell's original correspondence \cite[Chapter 25]{mordelldiophantine} when $K = \Q$ and $S$ contains only the infinite place. This subsection is not needed for proving the main theorems in this paper, but we include it to give a more modern interpretation of Mordell's work.

We show that $S$-integer-matrix binary quartic forms $f(X,Y)$ with an $S$-integral representation of $1$ (i.e., with $p,q \in \O_{K,S}$ with $f(p,q)=1$) may be transformed, under the standard action of $\SL_2(\O_{K,S})$, to demonic $S$-integer-matrix binary quartics. An element $g \in \SL_2(\O_{K,S})$ acts on $f(X,Y)$ by linear transformations as in \eqref{eq:SL2action} and on $(p,q)$ satisfying $f(p,q) = 1$ by $(p,q) \cdot g$.

\begin{lemma} \label{lem:bqrep1}
There is a bijection between demonic $S$-integer-matrix binary quartics
	\begin{equation} \label{eq:monicandzero}
	X^4 + 6c X^2 Y^2 + 4d X Y^3 + e Y^4
	\end{equation}
with $c, d, e \in \O_{K,S}$ and $\SL_2(\O_{K,S})$-equivalence classes of triples $(f,p,q)$, where $f$ is an $S$-integer-matrix binary quartic form and $p,q \in \O_{K,S}$ with $f(p,q) = 1$.

Furthermore, restricting to demonic $S$-integer-matrix binary quartics where $d$ is divisible by $2$ and $c^2 - e$ is divisible by $4$ gives a bijection with triples $(f,p,q)$ where $I'(f)$ and $J'(f)$ are divisible by $4$.
\end{lemma}

\begin{proof}
Given an $S$-integer-matrix binary quartic form $f(X,Y)$ and $p,q \in \O_{K,S}$ with $f(p,q)=1$, because $p$ and $q$ must generate the unit ideal, there exist $\alpha, \beta \in \O_{K,S}$ with $\alpha p + \beta q = 1$. Since the action of $(\begin{smallmatrix} \alpha & \beta \\ -q & p \end{smallmatrix})$ takes $(p,q)$ to $(1,0)$, there exists an $\SL_2(\O_{K,S})$-transformation taking $f$ to a monic $S$-integer-matrix binary quartic form. Then ``completing the quartic'' (which is possible because of the coefficients of $4$ and $6$ for an $S$-integer-matrix form) shows that there exists a $\SL_2(\O_{K,S})$-transformation of $f$ giving a binary quartic of the form \eqref{eq:monicandzero}.

Given two binary quartics $f$ and $f'$ of the form \eqref{eq:monicandzero}, each with the representation $(p,q) = (1,0)$ of $1$, it is straightforward to check explicitly that there is no nontrivial element of $\SL_2(\O_{K,S})$ taking $(f,1,0)$ to $(f',1,0)$.

The last statement follows trivially since for the binary quartic \eqref{eq:monicandzero}, we compute $I' = 3c^2 + e$ and $J' = -c^3 - d^2 + ce$.
\end{proof}

Combining Lemma \ref{lem:bqrep1} with Theorem \ref{thm:Mordell}, we have the following:

\begin{corollary} \label{cor:Mordellrep1}
The following sets are in bijection:
\begin{enumerate}[label={\upshape(\roman*)}]
\item $S$-integral affine Weierstrass models $y^2 = x^3 + Ax + B$ of elliptic curves with $S$-integral points $(x_0,y_0)$ (namely, with $A, B, x_0, y_0 \in \O_{K,S}$),
\item binary quartics $X^4 + 6 c X^2 Y^2 + 8 d X Y^3 + e Y^4$ with $c,d,e \in \O_{K,S}$ and  $c^2 -e$ divisible by $4$,
\item $\SL_2(\O_{K,S})$-equivalence classes of triples $(f,p,q)$, where $f(X,Y)$ is an $S$-integer-matrix binary quartic form with $4 \mid I'(f)$ and $4 \mid J'(f)$ and $p,q \in \O_{K,S}$ with $f(p,q) = 1$.
\end{enumerate}
\end{corollary}

\section{Counting integral points on elliptic curves}

\subsection{Integral points and Selmer elements}

Let $E$ be an elliptic curve over $K$ with an integral affine Weierstrass model $\cE_{A,B}$ of the form \eqref{eq:Weierstrass}, and let $S$ be a finite set of primes of $K$. We consider the sequence of maps
	\begin{equation} \label{eq:EZtoSel2}
	\Psi \colon \cE_{A,B}(\O_{K,S}) \hookrightarrow E(K) \to E(K)/2E(K) \stackrel{\xi}{\hookrightarrow} \Sel_2(E/K)
	\end{equation}
where $\cE_{A,B}(\O_{K,S})$ denotes the $S$-integral points on $E_{A,B}$ and $\Sel_2(E/K)$ is the $2$-Selmer group of $E$ over $K$.

It is well known that elements of $\Sel_2(E/K)$ may be represented as binary quartic forms $f(X,Y)$ over $K$ such that the Jacobian of the associated genus one curve $C(f): Z^2 = f(X,Y)$ is isomorphic to $E$ and $C$ is locally soluble. More precisely, elements of $\Sel_2(E/K)$ are in bijection with $\PGL_2(K)$-equivalence classes of such binary quartic forms (see, e.g., \cite{bsd, arulmanjul-bqcount, coregular}). The $\PGL_2(K)$-action on binary quartic forms is induced from the following twisted action of $\GL_2(K)$ on binary quartics: for $g \in \GL_2(K)$ and a binary quartic $f(X,Y)$, we have $(g \cdot f)(X,Y) = (\det g)^{-2} f((X,Y) \cdot g)$. The ring of $\PGL_2(K)$-invariants is still the polynomial ring generated by $I$ and $J$.

The map $\xi: E(K)/2E(K) \hookrightarrow \Sel_2(E/K)$ sends a rational point $P \in E(K)$ to the rational binary quartic form arising from the degree $2$ map $E \to \P^1$ given by the divisor $O+P$ (as described in \S \ref{sec:Mordell}).
The composition $\Psi$ of the maps in \eqref{eq:EZtoSel2} is thus given by one direction of the bijection in Theorem \ref{thm:Mordell}, from an $S$-integral point $P = (x_0,y_0)\in \cE_{A,B}(\O_{K,S})$ to the $\PGL_2(K)$-equivalence class of the corresponding $S$-integer-matrix demonic binary quartic form $f_P(X,Y) := X^4 - 6 x_0 X^2 Y^2 + 8 y_0 X Y^3 + (-4A-3x_0^2) Y^4$. Note that the genus one curve $C(f_P)$ associated to such a form (in fact, any monic binary quartic form) is automatically globally soluble over $K$; indeed, $f_P(1,0) = 1$ gives a rational solution. This is not surprising since, by construction, the image of $P$ in $\Sel_2(E/K)$ lies in the subset of globally soluble forms, namely, the image of $E(K)/2E(K)$.

Writing $E(K)\cong \Z^{\rank{E(K)}}\oplus E(K)_{\mathrm{tors}}$, we see that $|E(K)/2E(K)| \leq 4 \cdot 2^{\rank{E(K)}}$. Hence the image of $\xi$, and thus the image of the composition map $\Psi$, is of size
$\leq 4 \cdot 2^{\rank{E(K)}}$.
Therefore, to prove Theorem \ref{thm:optimalgeneralbound}, it suffices to show that the size of each fiber of the map $\Psi$ is bounded as follows:

\begin{proposition}\label{prop:the fiber bound}
Let $K$ be a number field and $S$ a finite set of places of $K$ containing all infinite places of $K$.
Let $f(X,Y) = X^4 + a_2 X^2 Y^2 + a_3 X Y^3 + a_4 Y^4 \in \O_{K,S}[X,Y]$ be a demonic $S$-integer-matrix binary quartic form such that the discriminant $\Delta(f)$ is squarefree in $\O_{K,S}$. Then
$$\left| \{\gamma\in \PGL_2(K) \textrm{ such that } \gamma\cdot f \textrm{ is demonic and $S$-integer-matrix}\} \right|
\leq C^{2|S|+1} |\Cl(\O_{K,S})[2]|,$$
where $C = 7^{2^7}$.
\end{proposition}

Note that the condition that $\Delta(f)$ is squarefree can be arranged by enlarging $S$; this condition is chosen to coincide with the hypothesis in Theorem \ref{thm:optimalgeneralbound} that $S$ contains all $\pfrak$ for which $\pfrak^2 \mid \Delta(f)$.

\subsection{The fiber bound}
To prove Proposition \ref{prop:the fiber bound}, we first establish properties of any $\gamma \in \PGL_2(K)$ that sends a demonic binary quartic form $f$ to another demonic form.
We then show that each such $\gamma$ gives rise to a solution of a Thue--Mahler equation, and invoke the work of Evertse \cite{evertse-on-equations-in-s-units-and-the-thue-mahler-equation} that the number of such solutions is bounded.

\begin{lemma} \label{lem:gammaprops}
Let $K$ be a number field and $S$ a finite set of places of $K$ containing all infinite places of $K$.
Let $f(X,Y) = X^4 + a_2 X^2 Y^2 + a_3 X Y^3 + a_4 Y^4 \in \O_{K,S}[X,Y]$ be a demonic $S$-integral binary quartic form.
For any $\gamma\in \PGL_2(K)$ such that $\gamma\cdot f$ is demonic and $S$-integer-matrix, write $\gamma = \left(\begin{smallmatrix} a & b\\ c& d\end{smallmatrix}\right)$ with $a,b,c,d\in \O_{K,S}$.  Then we have
	\begin{enumerate}[label={\normalfont (\roman*)}]
	\item $(a,b) = (a,b,c,d)$ as ideals of $\O_{K,S}$, and
	\item $f(a,b) = (\det \gamma)^2$ divides $\Delta(f) (a,b)^4$ in $\O_{K,S}$.
	\end{enumerate}
\end{lemma}

\noindent Lemma \ref{lem:gammaprops} follows from localizing and proving a simpler version for principal ideal domains:

\begin{lemma} \label{lem:gammapropsPID}
Let $R$ be a principal ideal domain with field of fractions $k$. Let $f(X,Y) = X^4 + a_2 X^2 Y^2 + a_3 X Y^3 + a_4 Y^4 \in R[X,Y]$ be a demonic $R$-integer-matrix binary quartic form. Let $\gamma\in \PGL_2(k)$ such that $\gamma\cdot f$ is demonic, $R$-integer-matrix, and also in $R[X,Y]$. Write $\gamma =: \left(\begin{smallmatrix} a & b\\ c& d\end{smallmatrix}\right)$ with $a,b,c,d\in R$ and $(a,b,c,d) = (1)$ as ideals of $R$. Then we have
	\begin{enumerate}[label={\normalfont (\roman*)}]
	\item $(a,b) = (1)$ as ideals of $R$, and \label{part:1ofgammapropsPID}
	\item $f(a,b) = (\det \gamma)^2$ divides the discriminant $\Delta(f)$ of $f$ in $R$. \label{part:2ofgammapropsPID}
	\end{enumerate}
\end{lemma}
\noindent Here $f\in R[X,Y]$ is called {\em $R$-integer-matrix} if $4$ divides the coefficients of $X^3 Y$ and $X Y^3$ in $f$ and $6$ divides the coefficient of $X^2 Y^2$ in $f$ as elements of $R$.

\begin{proof}
Write $(a,b) = (g)$, so there exist $\alpha, \beta \in R$ with $a = g \alpha$ and $b = g \beta$. Since $(\alpha, \beta) = (1)$, there exist $\widetilde{\alpha}, \widetilde{\beta} \in R$ such that $\alpha\widetilde{\alpha} - \beta\widetilde{\beta} = 1$.
Let $\widetilde{\gamma} := \left(\begin{smallmatrix} \alpha & \beta\\ \widetilde{\beta}& \widetilde{\alpha}\end{smallmatrix}\right)\in \SL_2(R)$, and let $\eta := c\widetilde{\alpha} - d\widetilde{\beta}\in R$. Define
$$U := \gamma\widetilde{\gamma}^{-1} = \left(\begin{array}{cc} a & b\\ c& d\end{array}\right)\cdot \left(\begin{array}{cc} \widetilde{\alpha} & -\beta\\ -\widetilde{\beta}& \alpha\end{array}\right) = \left(\begin{array}{cc} g & 0\\ \eta& \frac{\det{\gamma}}{g}\end{array}\right),$$ 
implying that
$\gamma = U \widetilde{\gamma}.$

We now show that $g$ divides all the entries of $U$, namely, $g^2 \mid \det{\gamma}$ and $g \mid \eta$. 
Let $\widetilde{f} := \widetilde{\gamma}\cdot f\in R[X,Y]$, with no twisting necessary since $\widetilde{\gamma}\in \SL_2(R)$. Note that $\widetilde{f}$ is $R$-integer-matrix, since the property of being integer-matrix is preserved by the action of $\SL_2(R)$. Write
$$\widetilde{f}(X,Y) =: \widetilde{a}_0 X^4 + \widetilde{a}_1 X^3 Y + \widetilde{a}_2 X^2 Y^2 + \widetilde{a}_3 X Y^3 + \widetilde{a}_4 Y^4\in R[X,Y].$$
Then $(\gamma\cdot f)(X,Y) = (\det \gamma)^{-2} (U\cdot \widetilde{f})(X,Y) = (\det \gamma)^{-2}  \widetilde{f}\left(gX + \eta Y, \frac{\det{\gamma}}{g} Y\right)$. Expanding, we compute that the $X^4$-coefficient in $\gamma\cdot f$ is
$$(\gamma\cdot f)(1,0) = f(a,b) = g^4 f(\alpha,\beta) = \frac{g^4 \widetilde{a}_0}{(\det{\gamma})^2}.$$
Since it is also $1$ by hypothesis, we find that $\frac{(\det{\gamma})^2}{g^4} = \widetilde{a}_0\in R$. Thus $g^4$ divides $(\det{\gamma})^2$, so $g^2$ divides $\det \gamma$.
Now the $X^3 Y$-coefficient of $\gamma \cdot f$ is
$$\frac{4 g^3\eta \widetilde{a}_0 + g^2 (\det{\gamma}) \widetilde{a}_1}{(\det{\gamma})^2} = 0.$$ Substituting for $\widetilde{a}_0$, we find that $\widetilde{a}_1 = -4 \frac{(\det{\gamma}) \eta}{g^3} \in 4 R$ (since $\widetilde{f}$ is $R$-integer-matrix).
Finally, the $X^2 Y^2$-coefficient of $\gamma \cdot f$ is
$$\frac{6 g^2 \eta^2 \widetilde{a_0} + 3 g \eta (\det{\gamma}) \widetilde{a}_1 + (\det{\gamma})^2 \widetilde{a}_2}{(\det{\gamma})^2} = - \frac{6\eta^2}{g^2} + \widetilde{a}_2$$
after substituting for $\widetilde{a}_0$ and $\widetilde{a}_1$.
Since $\widetilde{a}_2\in 6 R$ and this coefficient lies in $6 R$ as well (since both are $R$-integer-matrix), we deduce that $g^2$ divides $\eta^2$, so $g$ divides $\eta$.

Since $g$ divides all the entries of $U$, we see that $g$ divides all the entries of $U\cdot \widetilde{\gamma} = \gamma$, implying that $g$ divides $(a,b,c,d) = (1)$, whence $(g)=(1)$, proving \ref{part:1ofgammapropsPID}.

Now since $g\in R^\times$ it follows that
$\det{\gamma}$ divides $\widetilde{a}_1 = -4 \eta \det \gamma$, and of course $(\det \gamma)^2$ divides $\widetilde{a}_0.$
We thus find that $(\det\gamma)^2$ divides $\Delta(\widetilde{f}) = \Delta(f)$ (since every term of $\Delta(\widetilde{f})$ is a multiple of either $\widetilde{a}_0$ or $\widetilde{a}_1^2$), proving \ref{part:2ofgammapropsPID}.
\end{proof}

\begin{proof}[Proof of Lemma \ref{lem:gammaprops}]
Given the binary quartic form $f(X,Y)$ and $\gamma \in \PGL_2(K)$ as in the lemma, it suffices to show the claim upon localizing. For any prime $\p\not\in S$ of $K$, we may almost directly apply Lemma \ref{lem:gammapropsPID} with $R := \O_{K,S\cup \{\p\}}$ and $k := \Frac{R}$; the only difficulty is that $(a,b,c,d)_\p := (a,b,c,d)\cdot R$ is not necessarily the unit ideal in $R$. However, since $R$ is a discrete valuation ring, the ideal $(a,b,c,d)_\p$ is a nonzero principal ideal $(\delta)$, so we need only divide $a, b, c, d$ by $\delta$ to apply Lemma \ref{lem:gammapropsPID}. Then since $f$ has degree $4$, we obtain $f(a,b) = (\det \gamma)^2 \delta^4$. The desired result for $K$ follows.
\end{proof}

\begin{remark}\label{what squarefreeness buys}
If $\Delta(f)$ is squarefree in $\O_{K,S}$, in the setup of Lemma \ref{lem:gammaprops}, we in fact obtain $(f(a,b)) = (a,b)^4$ as ideals. Indeed, since $f$ is a homogeneous quartic, we have $(a,b)^4 \mid (f(a,b))$, and since $(\det \gamma)^2 = (f(a,b)) \mid (\Delta(f)) (a,b)^4$ and $(\Delta(f))$ is squarefree by hypothesis, we have $(f(a,b)) \mid (a,b)^4$. By unique factorization of ideals, we thus see that $$(\det \gamma) = (a,b)^2,$$ which implies that the ideal $(a,b)$ represents a $2$-torsion class in $\Cl(\O_{K,S})$. This is the source of the factor of $|\Cl(\O_{K,S})[2]|$ in the bound of Proposition \ref{prop:the fiber bound}.
\end{remark}

To prove Proposition \ref{prop:the fiber bound}, the key idea is to relate the fibers of $\Psi$ to $S$-integral solutions to Thue-Mahler equations; the number of such solutions is bounded by work of Bombieri, Bombieri--Schmidt, Evertse, and many others. In particular, we use the following theorem of Evertse:

\begin{theorem}[{\cite[Theorem 3]{evertse-on-equations-in-s-units-and-the-thue-mahler-equation}}]
\label{thm:evertse's theorem} 
Let $K$ be a number field and $\Sigma$ a finite set of places of $K$ containing all infinite places. Let $F\in \O_{K,\Sigma}[x,y]$ be a homogeneous polynomial with at least three distinct roots in $\P^1(\Qbar)$. Then
$$\left| \{(a,b) \in \O_{K,\Sigma}^2 : F(a,b)\in \O_{K,\Sigma}^\times \} / \O_{K,\Sigma}^\times \right|
\leq 7^{(\deg F)^3([K:\Q]+2|\Sigma|)}.$$
\end{theorem}

\begin{proof}[Proof of Proposition \ref{prop:the fiber bound}]
Recall that $\Cl(\O_{K,S})$ is a quotient of $\Cl(\O_K)$. Let $P$ be a set of prime ideals of $\O_K$ for which the canonical map $P\to \Cl(\O_{K,S})$, taking an element of $P$ to its ideal class, is a bijection onto the $2$-torsion subgroup $\Cl(\O_{K,S})[2]$ of $\Cl(\O_{K,S})$. Such a set of representatives exists by Chebotarev's density theorem applied to the Hilbert class field of $K$.

Lemma \ref{lem:gammaprops} shows that for any $\gamma \in \PGL_2(K)$ (represented by $\left(\begin{smallmatrix} a & b\\ c & d\end{smallmatrix}\right)$ with $a,b,c,d\in \O_{K,S}$) such that $\gamma \cdot f$ is demonic and $S$-integer-matrix, we have that $(a,b) = (a,b,c,d)$ as ideals of $\O_{K,S}$ and $f(a,b) = (\det{\gamma})^2$ is a square dividing $\Delta(f)\cdot (a,b,c,d)^4$ in $\O_{K,S}$. Further, by Remark \ref{what squarefreeness buys}, since $\Delta(f)$ is squarefree in $\O_{K,S}$, we have that $(a,b,c,d)^2 = (a,b)^2 = (\det \gamma)$.

Therefore, by the definition of $P$, there is a prime ideal $\p\in P$ and an $\alpha\in K^\times$ for which $\p = \alpha (a,b,c,d)$. Since $\p\subseteq \O_K\subseteq \O_{K,S}$ it follows that $\alpha a, \alpha b, \alpha c, \alpha d\in \O_{K,S}$. Scaling each of $a,b,c,d$ by $\alpha$ (which does not change $\gamma\in \PGL_2(K)$) we may without loss of generality assume that $(a,b,c,d) = (a,b) = \p$ in $\O_{K,S}$.

Thus, given $\gamma\in \PGL_2(K)$ for which $\gamma\cdot f$ is both demonic and $S$-integer-matrix, we get a pair $(a,b)\in \O_{K,S}^2$, well defined up to the action of $\O_{K,S}^\times$ (since $\gamma$ is an equivalence class of matrices in $\GL_2(K)$ modulo scaling by $K^\times$, and we have pinned down the ideal $(a,b)$ via $(a,b) = \p$ with $\p\in P$).

We now claim that the map
\begin{align}
\Phi \colon \{\gamma\in \PGL_2(K) &: \gamma\cdot f \textrm{ is demonic and $S$-integer-matrix}\} \label{eq:Phidef}
\\&\to \bigcup_{\p\in P} \{(a,b)\in \O_{K,S}^2 \mid (a,b) = \p, (f(a,b)) = \p^4\}/\O_{K,S}^\times, \nonumber
\end{align}
taking $\gamma$ as above to the equivalence class of $(a,b)$, defined as above, is injective.

Indeed, if $\gamma,\gamma'\in \PGL_2(K)$ map to the same $(a,b)\in \O_{K,S}^2 / \O_{K,S}^\times$, write $\gamma = \left(\begin{smallmatrix} a & b\\ c& d\end{smallmatrix}\right)$ and $\gamma' = \left(\begin{smallmatrix} a & b\\ c'& d'\end{smallmatrix}\right),$ and note that
$$
\gamma'\gamma^{-1} = \left(\begin{array}{cc} 1 & 0\\ \frac{c'd - cd'}{\det{\gamma}}& 1\end{array}\right).
$$ 
Let $\lambda := \frac{c'd - cd'}{\det{\gamma}}\in K$. Since
$$(\gamma'\cdot f)(X,Y) = ((\gamma'\gamma^{-1}) (\gamma\cdot f))(X,Y) = (\gamma\cdot f)(X + \lambda Y, Y)$$
and both $\gamma\cdot f$ and $\gamma'\cdot f$ are demonic by hypothesis, it follows that $\lambda = 0$ and so $\gamma = \gamma'$, as desired.

Thus, the size of the domain of $\Phi$ is bounded above by the size of the codomain of $\Phi$, which we now bound. Write
$$M_\p := \{(a,b)\in \O_{K,S}^2 \mid (a,b) = \p, (f(a,b)) = \p^4\}.$$
Thus the codomain is $\bigcup_{\p\in P} M_\p / \O_{K,S}^\times$, and we will compute an upper bound for each term $M_\p / \O_{K,S}^\times$.

First, note that the canonical map $M_\p / \O_{K,S}^\times \to M_\p / \O_{K,S\cup \{\p\}}^\times$, taking equivalence classes of elements of $M_\p$ modulo the diagonal action of $\O_{K,S}^\times$ to equivalence classes modulo the action of the larger $\O_{K,S\cup \{p\}}^\times$, is in fact a bijection. This is simply because given $\alpha\in K^\times$ and $(a,b)\in \O_{K,S}$ for which $(a,b) = \alpha (a,b) = \p$ as ideals of $\O_{K,S}$, we have that $\alpha\in \O_{K,S}^\times$.

Next we enlarge $M_\p$ as follows. 
Since $(f(a,b)) = \p^4$ implies that $f(a,b)\in \O_{K,S\cup \{\p\}}^\times$, we observe that
$$M_\p\subseteq \left\{(a,b)\in \O_{K,S\cup \{\p\}}^2 : f(a,b)\in \O_{K,S\cup \{p\}}^\times\right\},$$
and hence
$$M_\p / \O_{K,S\cup \{\p\}}^\times\subseteq \left\{(a,b)\in \O_{K,S\cup \{\p\}}^2 : f(a,b)\in \O_{K,S\cup \{p\}}^\times\right\} / \O_{K,S\cup \{p\}}^\times.$$
Applying Theorem \ref{thm:evertse's theorem} here with $F=f$ and $\Sigma = S \cup \{\p\}$, and using $[K:\Q] \leq 2|S|$, we find
\begin{align*}
\left| M_\p / \O_{K,S}^\times \right|
&= \left| M_\p / \O_{K,S\cup \{\p\}}^\times \right| \\
& \leq \left| \{(a,b)\in \O_{K,S\cup \{\p\}}^2 : f(a,b)\in \O_{K,S \cup \{\p\}}^\times\} / \O_{K,S\cup \{\p\}}^\times \right| \\
&\leq C^{2|S|+1},
\end{align*}
where $C = 7^{2^7}$.

To bound the codomain of $\Phi$, we sum this over $\p\in P$ and use $|P| = |\Cl(\O_{K,S})[2]|$ to obtain
\begin{equation*}
    |\{\gamma\in \PGL_2(K) : \gamma\cdot f \textrm{ is demonic and $S$-integer-matrix}\}| \leq C^{2|S|+1} |\Cl(\O_{K,S})[2]|.
\qedhere
\end{equation*}
\end{proof}
 
When $K = \Q$, the fiber bound may be improved as follows by combining arguments of Bombieri--Schmidt \cite{bombierischmidt} with Evertse's bounds \cite{evertse-on-equations-in-s-units-and-the-thue-mahler-equation} used in the proof of Proposition \ref{prop:the fiber bound}:

\begin{proposition}\label{prop:fiberboundoverQ}
Let $f(X,Y) = X^4 + a_2 X^2 Y^2 + a_3 X Y^3 + a_4 Y^4 \in \Z[X,Y]$ be a demonic binary quartic form. The number of elements $\gamma\in \PGL_2(\Q)$ such that $\gamma\cdot f$ is demonic is
$$\ll \prod_{p^2 \mid \Delta(f)} \min \left\{4 \floor{\frac{v_p(\Delta(f))}{2}} + 1,\, 7^{2^7} \right\}.$$
\end{proposition}

\begin{proof}
By the same argument as in the proof of Proposition \ref{prop:the fiber bound}, here using just Lemma \ref{lem:gammapropsPID} with $R=\Z$, we reduce to bounding
\begin{equation} \label{eq:codomainsize}
\sum_{\delta^2 \mid \Delta(f)} \# \{(a,b)\in \Z^2 : \gcd(a,b) = 1, f(a,b) = \delta^2\}.
\end{equation}

We divide the set of primes $p^2 \mid \Delta(f)$ into two sets to obtain a hybrid bound.
Let $T$ be the set\footnote{Taking $T$ to be the empty set in this argument gives a weaker but simpler upper bound for \eqref{eq:codomainsize}, and thus for Proposition \ref{prop:the fiber bound}, of
$\displaystyle \prod_{p^2 \mid \Delta(f)} \left(4 \floor{\frac{v_p(\Delta(f))}{2}} + 1\right).$} of primes such that 
$C = 7^{2^7} \leq 4 \floor{\frac{v_p(\Delta(f))}{2}} + 1$, and let
$$D := \prod_{\substack{p^2\mid \Delta(f) \\ p\not\in T}} p^{v_p(\Delta(f))}.$$
Given $\delta$ such that $\delta^2\mid \Delta(f)$, set $\nu := \gcd(\delta, D)$ and $\mu := \frac{\delta}{\nu}$.

The argument of Bombieri--Schmidt in \cite[Section VI]{bombierischmidt}, specifically Lemma 7 and the second-to-last paragraph, produces $\leq 4^{\omega(\nu)}$ many quartic forms $f_{\nu, i}$, depending only on $f$ and $\nu$, such that the number of relatively prime solutions to $f(a,b) = \delta^2 =  \nu^2 \mu^2$ is bounded above by the sum of the numbers of relatively prime solutions of $f_{\nu, i}(a,b) = \mu^2$. Rewriting \eqref{eq:codomainsize} in terms of $\mu$ and $\nu$ gives
\begin{align}
\sum_{\delta^2 \mid \Delta(f)} &\# \{(a,b)\in \Z^2 : \gcd(a,b) = 1, f(a,b) = \delta^2\} \nonumber \\
&= \sum_{\nu^2\mid D}\sum_{\mu^2\mid \frac{\Delta(f)}{D}} \# \{(a,b)\in \Z^2 : \gcd(a,b) = 1, f(a,b) = \mu^2 \nu^2\}
\nonumber \\
&\leq \sum_{\nu^2\mid D}\sum_{i=1}^{4^{\omega(\nu)}} \sum_{\mu^2\mid \frac{\Delta(f)}{D}} \# \{(a,b)\in \Z^2 : \gcd(a,b) = 1, f_{\nu, i}(a,b) = \mu^2\}. \label{eq:expandsum}
\end{align}

To bound the number of solutions to $f_{\nu,i}(a,b) = \mu^2$, we use Theorem \ref{thm:evertse's theorem} with $K = \Q$, $\Sigma = T \cup \{\infty\}$, and $F = f_{\nu, i}$. The innermost sum of \eqref{eq:expandsum} is at most
$$
\# \{(a,b)\in \Z^2 : \gcd(a,b) = 1, f_{\nu, i}(a,b)\in \Z[T^{-1}]^\times = \O_{\Q,T}^\times\} \leq C^{|T|+3/2}.
$$
Hence the size \eqref{eq:codomainsize} of the codomain of $\Phi$ is
\begin{align*}
\sum_{\delta^2 \mid \Delta(f)} \# \{(a,b)\in \Z^2 : \gcd(a,b) = 1, f(a,b) = \delta^2\} 
&\ll \sum_{\nu^2\mid D} \sum_{i=1}^{4^{\omega(\nu)}} C^{|T|}
\\&= C^{|T|} \sum_{\nu^2\mid D} 4^{\omega(\nu)}
\\&= C^{|T|} \prod_{p^2\mid D} \left(4 \floor{\frac{v_p(\Delta(f))}{2}} + 1\right). \qedhere
\end{align*}
\end{proof}

\begin{remark} \label{rmk:fourtotwo}
When $2\leq v_p(\Delta(f)) < 4$, evidently $p\not\in S$, and either $p\nmid \nu$ (in which case there is no factor corresponding to $p$) or $v_p(\nu) = 2$. By simply enumerating cases of $f$ over $\Q_p$ one finds that the number of disks required for \cite[Lemma 7]{bombierischmidt} is in fact at most $3$, because at least two roots lie in the same residue disk modulo $p$. This translates into an improvement in the factor for $p$ in Theorem \ref{thm:the bound on the number of integral points} to $4$ (rather than $5$).
\end{remark}

\subsection{Bounds on the number of integral points on an elliptic curve}

Combining Proposition \ref{prop:the fiber bound} with the bound on the image of $\Psi$ gives Theorem \ref{thm:optimalgeneralbound} immediately.

\begin{proof}[Proof of Theorem \ref{thm:optimalgeneralbound}]
We showed that the map $\Psi \colon \cE_{A,B}(\O_{K,S})\to E(K)/2 E(K)\subseteq \Sel_2(E/K)$, taking an integral point of $\cE_{A,B}$ to the $\PGL_2(K)$-equivalence class of its corresponding binary quartic form $f$ (by Theorem \ref{thm:Mordell}), has image of size $\ll 2^{\rank{E(K)}}$. Recall that the binary quartic $f$ has invariants $I = -48A$ and $J = -1728B$, so $\Delta(f) =
2^{8} \Delta_{A,B}$. Applying Proposition \ref{prop:the fiber bound} bounds the size of a fiber of the map $\Psi$, and combining the two estimates gives the theorem.
\end{proof}
Theorem \ref{thm:the bound on the number of integral points} follows from Proposition \ref{prop:fiberboundoverQ} in exactly the same way.

\section{Bounding moments of the number of integral points on elliptic curves} \label{sec:mainthmpf}
Theorem \ref{thm:the bound on the moments}, for $\moment < \log_2 5$, follows from ``averaging'' the bound in Theorem \ref{thm:the bound on the number of integral points} and analytic techniques. The additional crucial input is  Bhargava--Shankar's result that the average size of the $5$-Selmer group of elliptic curves in $\Funiv$, ordered by height, is bounded (proved over $\Q$ in \cite[Theorem 31]{arulmanjul-5Sel} and over global fields in \cite{BSW-globalfields2}):

\begin{theorem}[{\cite{BSW-globalfields2}}]\label{the five selmer bound}
Let $K$ be a number field. Then the average size of the $5$-Selmer group of elliptic curves in $\Funiv(\O_K)$, ordered by height, is
$$\Avg_{\cE_{A,B} \in \Funiv(\O_K)} |\Sel_5(\cE_{A,B})| = 6.$$
\end{theorem}

\begin{remark} \label{rmk:otherheights}
As mentioned in the introduction, the family $\Funiv(\O_K)$ is ordered by the height defined by \eqref{eq:ourheightdefinition}, in both Theorem \ref{the five selmer bound} and in our averaging results. A slightly different height $\widetilde{H}$, which coincides with the usual height on weighted projective space $\P(4,6)$, may also be used, as \cite{BSW-globalfields2} also proves that the average $5$-Selmer size is $6$ for short Weierstrass elliptic curves $E$ over $K$ ordered by $\widetilde{H}$. Since $\widetilde{H}$ is invariant under scaling, i.e., all integral models of a given elliptic curve $E$ over $K$ have the same value of $\widetilde{H}$, we must fix a quasi-minimal integral model for each curve in order to count integral points. Our results and proofs below may be modified appropriately to show that the $\moment$-th moment of the number of $S$-integral points on (quasi-minimal integral models of) elliptic curves over $K$, ordered by $\widetilde{H}$, is bounded, for $0 < \moment \leq \log_2 5$. 
\end{remark}

\subsection{Moments of the number of integral points in families of elliptic curves} \label{sec:averages}

We now prove the following slightly weaker version of Theorem \ref{thm:the bound on the moments}; the remaining case where $\moment = \log_2 5$ will be handled in Section \ref{sec:weights}.

\begin{theorem} \label{thm:momentsbound}
Let $\FF\subseteq \Funiv(\O_K)$ be a subset of positive lower density (ordering by height). Let $S$ be a finite set of places of $K$. Let $0 < \moment < \log_2 5 = 2.3219 \ldots$. Then
$$\Avg_{\FF}( |\cE_{A,B}(\O_{K,S})|^\moment ) \ll_{\moment, \FF} \left( C^{2|S|} \left| \Cl(\O_{K,S})[2] \right| \right)^\moment$$
where the average is taken over all elliptic curves $\cE_{A,B}\in \FF$ ordered by height and $C = 7^{2^7}$.
\end{theorem}

\begin{proof}
We may immediately reduce to the case of $\FF = \Funiv(\O_K)$ because 
$$\sum_{\cE_{A,B}\in \FF^{\leq T}} |\cE_{A,B}(\O_{K,S})|^\moment\leq \sum_{\cE_{A,B}\in \Funiv^{\leq T}(\O_K)} |\cE_{A,B}(\O_{K,S})|^\moment$$
and, by the positive lower density hypothesis,
$$\sum_{\cE_{A,B}\in \FF^{\leq T}} 1\gg_{\FF} \sum_{\cE_{A,B}\in \Funiv^{\leq T}(\O_K)} 1.$$ Thus, we may now assume that $\FF = \Funiv(\O_K)$.

Applying Theorem \ref{thm:optimalgeneralbound} immediately gives
\begin{align*}
\Avg_{\Funiv(\O_K)}&\left( |\cE_{A,B}(\O_{K,S})|^\moment \right) \\
&\ll C^{(2|S| + 1) \moment} |\Cl(\O_{K,S})[2]|^\moment \Avg_{\Funiv(\O_K)}\left((2^\moment)^{\rank{E(K)}} O(1)^{\omega_{\geq 2}(\Delta_{A,B})}\right),
\end{align*}
where $C = 7^{2^7}$ as before.
Thus it suffices to show that $$\Avg_{\Funiv(\O_K)}\left((2^\moment)^{\rank{E(K)}} O(1)^{\omega_{\geq 2}(\Delta_{A,B})}\right)\ll O(1)^{(\log_2{5} - \moment)^{-1}}.$$
Set $\eps := \frac{\log_2{5}}{\moment} - 1 > 0.$ By H\"{o}lder's inequality with dual exponent pair $(1 + \eps, 1 + \eps^{-1})$, we obtain
\begin{align}
\Avg_{\Funiv(\O_K)}&\left((2^\moment)^{\rank{E(K)}} O(1)^{\omega_{\geq 2}(\Delta_{A,B})}\right) \nonumber \\
&\ll \Avg_{\Funiv(\O_K)}\left((2^{\moment (1 + \eps)})^{\rank{E(K)}}\right)^{\frac{1}{1 + \eps}} \Avg_{\Funiv(\O_K)}\left(O(1)^{\omega_{\geq 2}(\Delta_{A,B}) (1 + \eps^{-1})}\right)^{\frac{\eps}{1 + \eps}} \label{eq:afterHolder} \\
&\ll \Avg_{\Funiv(\O_K)}\left(5^{\rank{E(K)}}\right) \Avg_{\Funiv(\O_K)}\left(\left(O(1)^{(\log_2{5} - \moment)^{-1}}\right)^{\omega_{\geq 2}(\Delta_{A,B})}\right)^\eps. \label{eq:productinequality}
\end{align}

We apply Theorem \ref{the five selmer bound} to bound the first average in \eqref{eq:productinequality} by $O(1)$. To bound the second average, we would like to apply Lemma \ref{boundedness of averages of ok weights} below with the discriminant polynomial $\Delta(A,B)$ on elliptic curves $\cE_{A,B} \in \Funiv(\O_K)$ with $H(A,B) \leq X$, $\eta = 1$, and $\lambda = X$. (Note that the weights $\vec{\alpha}$ are irrelevant in this case.)
In order to use Lemma \ref{boundedness of averages of ok weights}, we only need to verify its condition \ref{the probability hypothesis}, specifically that for a squarefree ideal $\dfrak \subseteq \O_K$ with $\Nm\,{\dfrak}\leq X^\delta$,
$$\mathop{\Prob}_{\substack{\cE_{A,B}\in \Funiv(\O_K) \\ H(A,B)\leq X}} \left(\dfrak^2\mid \Delta_{A,B} \right)
\ll O(1)^{\#\{\p\mid \dfrak\}} (\Nm\,{\dfrak})^{-2},$$
which we now check.

First, the number of solutions $(x,y)\in \O_K/\dfrak^2$ of $-16(4x^3 + 27y^2)\equiv 0\pmod{\dfrak^2}$ is
\begin{equation}\label{number of solutions mod prime squares}
\ll \Nm(\dfrak)^2 O(1)^{\#\{\pfrak\mid \dfrak\}}.
\end{equation}
Indeed, by the Chinese remainder theorem, it suffices to check \eqref{number of solutions mod prime squares} when $\dfrak$ is prime, and then upon fixing $x\in (\O_K/\dfrak)^\times$, we find a quadratic in $y$, whose number of solutions is bounded by $O(1)^{\#\{\pfrak\mid \dfrak\}}$ by Hensel lifting.

Also, the set $\{(A,B)\in \O_K^2 : H(A,B)\leq X\}$ covers each congruence class modulo $\dfrak^2$ roughly evenly, since given $(x,y)\in \O_K/\dfrak^2$, we see that
\begin{align}
\#&\left\{ (A,B)\in \O_K^2 :  H(A,B)\leq X \textrm{ and } (A,B)\equiv (x,y) \,(\mathrm{mod}\,\dfrak^2) \right\} \nonumber
\\&\qquad \ll \Nm(\dfrak)^{-4} \#\left\{(A,B)\in \O_K^2 : H(A,B)\leq X\right\}.
\label{equidistribution in congruence classes}
\end{align}
Combining these two observations yields
\begin{align}
\#&\left\{(A,B)\in \O_K^2 : H(A,B)\leq X, \Delta_{A,B}\equiv 0 \,(\mathrm{mod}\,\dfrak^2)\right\} \label{eq:toinsert} \\
&\qquad \ll \Nm(\dfrak)^{-2} O(1)^{\#\{\pfrak\mid \dfrak\}} \# \left\{(A,B)\in \O_K^2 : H(A,B)\leq X\right\}.  
\end{align}
We thus may apply Lemma \ref{boundedness of averages of ok weights} to bound the second average in \eqref{eq:productinequality}, thereby giving the theorem.
\end{proof}

We now prove the lemmas that will give an upper bound for the second average in \eqref{eq:productinequality}, namely, the average of $O(1)^{\omega_{\geq 2}(\Delta)}$ as $\Delta$ varies in a reasonable way. The first lemma shows that a small-norm divisor of the discriminant ideal may be used as a proxy for the discriminant itself when averaging.

\begin{lemma}\label{a small divisor controls everything}
Let $\delta$, $c$, $X$ be positive real numbers. Let $\zfrak\subseteq \O_K$ be an ideal with $\Nm(\zfrak)\leq X^c$. Then there exists a squarefree ideal $\dfrak$ such that $\dfrak^2 \mid \zfrak$ with
\begin{enumerate}[label={\upshape(\roman*)}]
\item \label{condition:dfrak1}
$\Nm(\dfrak)\leq X^\delta$, and 
\item \label{condition:dfrak2}
$\#\{\pfrak^2\mid \zfrak\} \leq (2 + \delta^{-1} c) \#\{\pfrak^2\mid \dfrak\}$. 
\end{enumerate}
\end{lemma}

\begin{proof}
Note that $\#\{\pfrak^2 \mid \zfrak : \Nm(\pfrak) > X^{\delta/2}\} \ll \delta^{-1} c$. Let $$\nfrak = \prod_{\substack{\pfrak^2 \mid \zfrak \\ \Nm\,{\pfrak}\leq X^{\delta/2}}} \pfrak,$$ so $\Nm(\nfrak)\leq X^{{c}/{2}}$ since $\nfrak^2\mid \zfrak$.
Thus if $\Nm(\nfrak)\leq X^\delta$, then taking $\dfrak = \nfrak$ suffices.

We may now assume that $\Nm(\nfrak) > X^\delta$. First, for every divisor $\nfrak'$ of $\nfrak$, we claim there exists an ideal $\dfrak'\mid \nfrak'$ for which
\begin{equation} \label{eq:condition-sizeofdprime}
\min(X^{\delta/2}, \Nm(\nfrak'))\leq \Nm(\dfrak')\leq X^\delta.
\end{equation}
This follows from the fact that $\Nm(\pfrak)\leq X^{\delta/2}$ for each $\pfrak\mid \nfrak$ by construction. We simply start with $\nfrak$ and remove one prime $\pfrak$ at a time until \eqref{eq:condition-sizeofdprime} is satisfied. We call this the ``deletion argument''.

An analogous argument shows that for every divisor $\nfrak'$ of $\nfrak$ for which $\Nm(\nfrak')\leq X^{\delta/2}$, there is an ideal $\dfrak'$ such that $\nfrak'\mid \dfrak'\mid \nfrak$ and $X^{\delta/2}\leq \Nm(\dfrak') \leq X^\delta$. Here, we instead add one prime $\pfrak$ at a time (which we can do because $\Nm\,{\nfrak} > X^\delta$) until the norm condition is satisfied; call this the ``insertion argument''.

Let $\dfrak$ be a divisor of $\nfrak$ for which $X^{\delta/2}\leq \Nm(\dfrak)\leq X^\delta$ (and thus satisfying \ref{condition:dfrak1}), for which $\#\{\pfrak\mid \dfrak\}$ is maximized subject to this condition, namely,
$$\dfrak := \mathrm{argmax} \left\{\#\{\pfrak\mid \dfrak'\} : X^{\delta/2}\leq \Nm(\dfrak')\leq X^\delta\right\}.$$ 
Note that we are maximizing over a nonempty finite set (by the deletion argument).

Let $\afrak$ be the ideal such that $\nfrak = \afrak \dfrak $. Because $\Nm(\nfrak) > X^\delta$ by assumption, we have $\Nm(\afrak) > 1$. By repeatedly applying the deletion argument to $\afrak$ as needed, we may decompose
\begin{equation} \label{eq:decomposition}
\nfrak =  \afrak_1 \cdots \afrak_N \dfrak
\end{equation}
with $X^{\delta/2}\leq \Nm(\afrak_i)\leq X^\delta$ for all $i < N$ and $\Nm(\afrak_N) < X^{\delta/2}$; note that $N \geq 1$ since $\afrak\neq (1)$. Since
$$X^{\delta (N-1)/2}\leq \Nm(\nfrak)\leq X^{{c/2}},$$
we conclude that $1\leq N\leq 1 + \delta^{-1} c$. Moreover, our choice of $\dfrak$ gives
$$\#\{\pfrak\mid \dfrak\} \geq \#\{\pfrak\mid \afrak_i\}$$
for \emph{all} $1 \leq i \leq N$: this follows for $i < N$ from the maximality imposed in the definition of $\dfrak$, and for $i = N$ from this maximality and the insertion argument. Hence the decomposition \eqref{eq:decomposition} gives
$$\#\{\pfrak\mid \nfrak\} \leq (N+1) \#\{\pfrak\mid \dfrak\},$$
so \ref{condition:dfrak2} is satisfied as well, which proves the lemma.
\end{proof}

We now establish a general lemma that is used both in Theorem \ref{thm:momentsbound} and in Section \ref{sec:weights}. Given a polynomial $\Delta(v_1,\ldots,v_n)$ whose behavior on a bounded set of $\vec{v}$ is somewhat controlled, we obtain an upper bound for the average size of a constant raised to the number of primes whose squares divide $\Delta(\vec{v})$, as $\vec{v}$ ranges over that set.

\begin{lemma}\label{boundedness of averages of ok weights}
Let $K$ be a number field. Let $\lambda$, $t$, $\eta$, $C_1$, $C_2$, and $C_3$ be positive real numbers, and suppose $\lambda \geq 1$. Let $N$ be a positive integer, and fix $\vec{\alpha} \in (\R^+)^N$ giving an action of $\R^+$ on $\vec{u} \in (K\otimes_\Q \R)^N$ by $\lambda\cdot {\vec{u}} := (\lambda^{\alpha_i} u_i)_i$.
Fix a bounded open subset $\mathcal{S} \subseteq (K\otimes_\Q \R)^N$.

Suppose $\Delta\in \O_K[X_1, \ldots, X_N]$ satisfies
\begin{enumerate}[label={\upshape(\alph*)}]
\item $\Delta(\lambda \cdot \vec{u})\leq C_1 \lambda^{C_2}$ for all $\vec{u}\in \mathcal{S}$, and
\item \label{the probability hypothesis} for all $\delta > 0$ with $\delta\ll_{\eta, \vec{\alpha}} 1$ and squarefree ideals $\afrak\subseteq \O_K$ with $\Nm(\afrak) \leq \lambda^\delta$,
\begin{equation}
\mathop{\Prob}_{\vec{v}\in \O_K^N\cap \lambda\cdot \mathcal{S}}\left(\afrak^2\mid \Delta(\vec{v})\right)
\leq C_3^{\#\{\pfrak\mid \afrak\}} (\Nm(\afrak))^{-1-\eta}.
\end{equation}
\end{enumerate}
Then
\[
\mathop{\Avg}_{\vec{v}\in \O_K^N\cap \lambda\cdot \mathcal{S}} \left(t^{\#\{\p^2\vert \Delta(\vec{v})\}}\right)\leq t^{O(1)} \zeta_K(1 + \eta)^{C_3 t^{O(1)}}
\]
where both $O(1)$'s depend on $\eta$, $\vec{\alpha}$, $C_1$, and $C_2$.
\end{lemma}

\begin{proof}
Fix $\delta\asymp_{\eta, \vec{\alpha}} 1$. For each $\vec{v}\in \O_K^N\cap \lambda\cdot \mathcal{S}$, we apply 
Lemma \ref{a small divisor controls everything} with $\zfrak = \Delta(\vec{v})$ to find a squarefree ideal $\dfrak(\vec{v})$ such that $\dfrak(\vec{v})^2 \mid (\Delta(\vec{v}))$ with
\begin{enumerate}[label={\upshape(\roman*)}]
    \item $\Nm\,{\dfrak(\vec{v})} \leq \lambda^\delta$, and
    \item \label{condition:dfrak2inlemma}
    $\#\{\pfrak^2\mid \Delta(\vec{v})\} \ll_{C_1, C_2} \delta^{-1} (\#\{\pfrak\mid \dfrak(\vec{v})\} + 1)$.
\end{enumerate}
In the following computation, every instance of $O(\cdot)$ is dependent on $C_1$ and $C_2$, which we omit for notational convenience:
\begin{align*}
\sum_{\vec{v}\in \O_K^N\cap \lambda\cdot \mathcal{S}} t^{\#\{\pfrak^2\mid \Delta(\vec{v})\}}
&\leq t^{O(\delta^{-1})}\sum_{\substack{\dfrak\text{ squarefree} \\ \Nm\,{\dfrak}\leq \lambda^\delta}} t^{O(\delta^{-1} \#\{\pfrak\mid \dfrak\})} \sum_{\substack{\vec{v}\in \O_K^N\cap \lambda\cdot \mathcal{S} \\ \dfrak^2 \mid \Delta(\vec{v})}} 1
& \text{by \ref{condition:dfrak2inlemma}}
\\
&\leq t^{O(\delta^{-1})} \left(\sum_{\Nm\,{\dfrak}\leq \lambda^\delta} \frac{(C_3 t^{O(\delta^{-1})})^{\#\{\pfrak\mid \dfrak\}}}{(\Nm\,{\dfrak})^{1+\eta}}\right)\left(\sum_{\vec{v}\in \O_K^N\cap \lambda\cdot \mathcal{S}} 1\right)
& \text{by \ref{the probability hypothesis}}
\\
&\leq t^{O(\delta^{-1})} \left(\sum_{\dfrak\subseteq \O_K} \frac{(C_3 t^{O(\delta^{-1})})^{\#\{\pfrak\mid \dfrak\}}}{(\Nm\,{\dfrak})^{1+\eta}}\right)\left(\sum_{\vec{v}\in \O_K^N\cap \lambda\cdot \mathcal{S}} 1\right)
\\
&\leq t^{O(\delta^{-1})} \zeta_K(1+\eta)^{C_3 t^{O(\delta^{-1})}} \sum_{\vec{v}\in \O_K^N\cap \lambda\cdot \mathcal{S}} 1,
\end{align*}
where the last inequality uses the fact that $(1+a)^b \geq 1 + ab$ for $a, b > 0$.
\end{proof}

\subsection{Families of elliptic curves with marked points} \label{sec:otherfamilies}

The arguments in Section \ref{sec:averages} may be modified appropriately to give averages or moments of the number of integral points on elliptic curves over $\Q$ in some other families for which we have finite upper bounds on the average $d$-Selmer group size for some $d \geq 2$. These include the families
\begin{align*}
\FF_1 &= \{y^2 + d_3 y = x^3 + d_2 x^2 + d_4 x \mid d_2, d_3, d_4 \in \Z, \, \Delta \neq 0 \} \qquad \text{and} \\
\FF_2 &= \{y^2+ d_1 x y + d_3 y = (x-d_2)(x-d_2')(x-d_2'') \mid d_1, d_2, d_2', d_2'', d_3 \in \Z, \, d_2 + d_2' + d_2'' = 0, \, \Delta \neq 0 \}
\end{align*}
of elliptic curves over $\Q$.
The family $\FF_1$ has a marked point at $(0,0)$, and the family $\FF_2$ has two marked points; these points are of infinite order and independent $100\%$ of the time \cite[Section 10]{cofreecounting}. The height $H(\cE)$ of a curve $\cE$ in these families is again a measure of the size of the coefficients, defined as $\max \left\{|d_i|^{\frac{12}{i}}\right\}$ (for $\FF_2$, we include $|d_2'|^6$ and $|d_2''|^6$ in that maximum). By \cite[Theorem 1.1]{cofreecounting}, the average size of the $3$-Selmer group (resp., $2$-Selmer group) in $\FF_1$ (resp., $\FF_2$), ordered by height, is bounded. We claim that the average number of integral points on the curves in these families is bounded, and in fact, a stronger statement holds:

\begin{theorem}\label{moment theorem for other families}
For any positive lower density family $\FF$ in $\FF_1$ (respectively, $\FF_2$) and any positive real number $\moment \leq \log_2 3 = 1.5850\ldots$ (resp., $\moment \leq 1$), we have
$$\Avg_\FF( \left| \cE(\Z) \right|^\moment ) \ll_{\moment,\FF} 1$$
where the average is taken over all $\cE\in \FF$ ordered by height.
\end{theorem}

\begin{proof}
We first explain the case of $\FF_1$. For $\moment < \log_2{3}$, the proof follows the same outline as that of Theorem \ref{thm:momentsbound}. As before, we reduce to the case of $\FF = \FF_1$. Let $\FFT := \{ \cE \in \FF : \Delta_{\cE} \neq 0 \textrm{ and } H(\cE) \leq T\}$ represent the curves in $\FF$ of height at most $T$. Given an integral model $\cE \in \FF$, let $E$ over $\Q$ be the corresponding elliptic curve. The bound of Theorem \ref{thm:the bound on the number of integral points} and H\"older's inequality give an inequality analogous to \eqref{eq:afterHolder}:
\begin{equation}
\sum_{\cE\in \FFT} |\cE(\Z)|^\moment
\ll \left(\sum_{\cE\in \FFT} (2^{\moment(1+\eps)})^{\rank{E(\Q)}}\right)^{\frac{1}{1+\eps}}
\left(\sum_{\cE\in \FFT} O(1)^{\omega_{\geq 2}(\Delta_{\cE})\left(1 + \eps^{-1}\right)}\right)^{\frac{\eps}{1+\eps}} \label{eq:afterholder2}
\end{equation}
The first term is bounded as before, by choosing $0 < \eps < \frac{\log_2 3}{s} - 1$ so that
$$
\left(2^{  \moment (1+\eps)} \right)^{\rank{E(\Q)}}\leq 3^{\rank{E(\Q)}} \leq \left|E(\Q)/3E(\Q)\right| \leq \left|\Sel_3(E)\right|.
$$
The bounds on the average $3$-Selmer size from \cite{cofreecounting} imply that
$$
\sum_{\cE\in \FFT} \left(2^{\moment(1+\eps) } \right)^{\rank{E(\Q)}} \ll \left| \FFT \right|.
$$
To bound the second term in \eqref{eq:afterholder2}, we again note that, by Lemma \ref{boundedness of averages of ok weights}, it suffices to check that \begin{equation}\label{the hypothesis we need to check in the other families proof}
\mathop{\Prob}_{\cE\in \FF^{\leq T}} \left(m^2\mid \Delta_{\cE} \right)\ll O(1)^{\omega(m)} m^{-2}
\end{equation}
when $m < T^\delta$ is squarefree.
First, for $\FF = \FF_1$ and $m < T^{\delta}$, each fiber of the natural reduction map $\FFT \to (\Z/m^2\Z)^3$ sending $E \in \FF_1$ to $(d_2, d_3, d_4)$ modulo $m^2$ is of size 
$$\ll \left(\frac{T^{1/6}}{m^2}+1\right)\left(\frac{T^{1/4}}{m^2}+1\right)\left(\frac{T^{1/3}}{m^2}+1\right) \ll \left|\FFT\right| m^{-6}.$$
Moreover, there are $\ll m^4 O(1)^{\omega(m)}$ solutions $(d_2, d_3, d_4)$ modulo $m^2$ to the discriminant vanishing modulo $m^2$. 
Indeed, by the Chinese remainder theorem, it suffices to verify this for $m = p$ prime. In that case, by Hensel's lemma, for each of the $\ll p^2$ many mod $p$ solutions with a nonvanishing differential, there are $p^2$ lifts to a mod $p^2$ solution, whereas there are $\ll p$ many mod $p$ solutions with a vanishing differential (and trivially $\leq p^3$ lifts of each to a mod $p^2$ solution).
We thus obtain the bound (\ref{the hypothesis we need to check in the other families proof}).

The argument for $\FF = \FF_2$ is entirely analogous, using the upper bound on the average size of the $2$-Selmer group for curves in $\FF_2$ from \cite{cofreecounting}. The equality cases $(\FF, \moment) = (\FF_1, \log_2 3)$ or $(\FF_2, 1)$ are proven in Theorem \ref{weighted two selmer average bound} below.
\end{proof}

\section{Arithmetic statistics with weights} \label{sec:weights}

We introduce a method to leverage the ``usual'' geometry-of-numbers techniques for proving bounds on average in arithmetic statistics to obtain bounds for {\em weighted} averages, provided the relevant weight function is sufficiently well behaved. Rather than formulate a general theorem, we apply the method to bound the weighted average of the number of $d$-Selmer elements of elliptic curves $\cE\in \FF$ ordered by height, with each $d$-Selmer element weighted by $O(1)^{\omega_{\geq 2}(\Delta_{\cE})}$, when $(\FF,d)\in \{(\Funiv(\O_K),5),  (\FF_1,3), (\FF_2,2)\}$.

These weighted averages then give the equality cases of $\moment = \log_2 d$ for Theorems \ref{thm:the bound on the moments} and \ref{thm:moments for families intro}. As previously noted, the equality case of most interest here is likely that of $\FF_2$, since it gives a bound for the average number of integral points on curves in $\FF_2$. However, we note that in fact these equality cases recover all of Theorems \ref{thm:the bound on the moments} and \ref{thm:moments for families intro}, since they imply that the $\moment$th moment for all $\moment < \log_2 d$ will also be bounded.

Let us first briefly sketch the key ideas in the unweighted counting method that we will be modifying. In some recent papers such as \cite{arulmanjul-bqcount, arulmanjul-tccount, arulmanjul-4Sel, arulmanjul-5Sel, cofreecounting}, the main goal is to count (average) the $d$-Selmer elements ($d = 2$, $3$, $4$, or $5$) of elliptic curves in either $\Funiv$ or other families; these Selmer elements are parametrized by orbits---with both global and local conditions---of a group $G(\Q)$ acting on a representation $V(\Q)$. It is possible to count the required rational orbits by counting (appropriately weighted) integral orbits, each of which corresponds to a lattice point in a fundamental domain. Davenport's Lemma roughly says that the number of such lattice points is the volume of the domain, or at least the well-behaved part of the domain; other methods are used to count the lattice points in the remaining cusps. One thereby obtains an asymptotic count of the number of lattice points in the fundamental domain, and after applying sieves to impose the necessary local conditions, these counts give the desired Selmer averages.

In this section, we upper bound the average number of $d$-Selmer elements with weights of $O(1)^{\omega_{\geq 2}(\Delta_\cE)}$. We interpret this weighted average as a weighted average of the relevant vectors in the lattice $V(\Z)$ (roughly, those in the fundamental domain), and the weights (in fact, the discriminants) depend only on the $G(\Q)$-invariants of the vector. This weighted average is an integral over the fundamental domain, and we split this integral into two pieces, depending on the size of the so-called torus parameters appearing in our description of points in the fundamental domain. For large torus parameters, we use a pointwise bound on the weight function to reduce to integrating the unweighted volumes over this region that is ``polynomially high in the cusp.'' For small torus parameters, we may use Davenport's Lemma to obtain enough equidistribution over congruence classes to apply Lemma \ref{boundedness of averages of ok weights}, which then bounds the integrand by the unweighted count.

\subsection{Weighted averages of $5$-Selmer elements of elliptic curves} \label{sec:weightuniv}

\begin{theorem}\label{the five selmer equality case}
Let $\FF\subseteq \Funiv(\O_K)$ be a subset of positive lower density (ordering by height). Let $S$ be a finite set of places of $K$ containing all infinite places. Let $\moment = \log_2 5 = 2.3219 \ldots$. Then
$$\Avg_{\FF}( |\cE_{A,B}(\O_{K,S})|^\moment) \ll_{\FF} \zeta_K(2)^{O(1)} O(1)^{|S|} |\Cl(\O_{K,S})|^\moment,$$
where the average is taken over all $\cE_{A,B}\in \FF$ ordered by height.
\end{theorem}

Theorem \ref{the five selmer equality case} follows immediately from the following weighted average bound, in exactly the same way as the beginning of the proof of Theorem \ref{thm:momentsbound}.

\begin{theorem}\label{weighted five selmer average bound}
$$\Avg_{\Funiv(\O_K)}\left(O(1)^{\omega_{\geq 2}(\Delta_{A,B})} |\Sel_5(E_{A,B}/K)|\right)\ll \zeta_K(2)^{O(1)}$$ where the average is taken over all $\cE_{A,B}\in \Funiv(\O_K)$ ordered by height.
\end{theorem}

\begin{proof}
We take $K = \Q$ for convenience; the argument is exactly the same for general number fields $K$ (since Lemma \ref{boundedness of averages of ok weights} is phrased in that generality), but when $K = \Q$, we may give precise citations to the intermediate results in \cite{arulmanjul-5Sel} rather than \cite{BSW-globalfields2}. So for the rest of the proof, let $K = \Q$. Let $\FunivX(\O_K)$ denote the elliptic curves $\cE_{A,B} \in \Funiv(\O_K)$ with $H(A,B) \leq X$, ordered by height.

To show the boundedness of this average, since $5^{\rank{E(K)}}\leq 1 + (|\Sel_5(E/K)| - 1)$ and because we have shown in the course of the proof of Theorem \ref{thm:momentsbound} that
$$\Avg_{\cE_{A,B} \in {\FunivX(\O_K)}}\left(O(1)^{\omega_{\geq 2}(\Delta_{A,B})}\right)\ll \zeta_K(2)^{O(1)},$$
it suffices to prove that
\begin{equation} \label{eq:nontrivial5selmeravg}
\Avg_{\cE \in \FunivX(\O_K)}\left(\sum_{1\neq \alpha\in \Sel_5(E/K)} O(1)^{\omega_{\geq 2}(\Delta(\alpha))}\right)
\end{equation}
is similarly bounded, i.e., where we only consider non-identity $5$-Selmer elements. 

Let $V(K) = K^5\otimes \wedge^2(K^5)$ and $G(K)$ be the subquotient of $\GL_5(K)^2$ defined by $\{(g_1,g_2) \in \GL_5(K) \times \GL_5(K) : \det(g_1)^2 \det g_2 = 1\} / \{(\lambda I_5, \lambda^{-2} I_5): \lambda \in K^\times\}$. Then elements of the $5$-Selmer group of an elliptic curve $\cE_{A,B}$ are in bijective correspondence with $G(K)$-orbits in the set of locally soluble $\vec{v} \in V(K)$ with invariants equal to $-3A$ and $-27B$ \cite{fishersadek-5sel}, and the usual discriminant $\Delta(\cE_{A,B})$ equals the discriminant $\Delta(\vec{v})$ of $\vec{v}$ under this correspondence. By relating rational and integral orbits, counting nontrivial $5$-Selmer elements can be translated into counting so-called {\em strongly irreducible} locally soluble $G(\Z)$-orbits of $V(\Z)$, as in \cite{arulmanjul-5Sel,BSW-globalfields2}.

For the rest of this proof, we adopt notation from \cite{arulmanjul-5Sel} as needed. For example, let $\mathcal{F}$ be the fundamental domain for the left action of $G(\Z)$ on $G(\R)$ written as $\{nak : n \in N'(a), a \in A', k \in K \}$, where $K = \SO_5(\R)^2$, $A' = \{a(s_1,\ldots,s_8) : s_i > c\}$ is a certain subset of pairs of diagonal matrices ($c > 0$ is an absolute constant), and $N'$ is a certain subset of pairs of lower triangular matrices, as defined in \cite[p.~8]{arulmanjul-5Sel}. Let $G_0 \subset G(\R)$ be a compact semialgebraic left $K$-invariant set that is the closure of an open nonempty set. Let $dh = dn\, d^*a\, dk$ be the Haar measure on $G(\R)$, suitably normalized, and set $C_{G_0} = 5 \int_{h \in G_0} dh$. Let $V^+$ and $V^-$ denote the subsets with positive and negative discriminant, respectively, and let $(V(\Z)^{\pm})^{\irr}$ denote the strongly irreducible elements in $V(\Z)^{\pm}$. For a subset $\mathcal{S}$ of $V(\Z)$, let $N(\mathcal{S};X)$ denote the number of $\vec{v} \in \mathcal{S}$ of height up to $X$.
Let $R^\pm$ be fundamental sets for the action of $G(\R)$ on $V^{\pm}(\R)$ as in \cite[(4) on p.~8]{arulmanjul-5Sel}, and $R^\pm(X)$ the subset of $R^\pm$ of height up to $X$. Finally, let $B^{\pm}(n,a;X)$ be the multiset $naG_0 \cdot R^{\pm}(X)$.

Now since nontrivial $5$-Selmer elements correspond to locally soluble strongly irreducible $G(\Q)$-orbits, whence fewer in number than strongly irreducible $G(\Z)$-orbits, we may upper bound \eqref{eq:nontrivial5selmeravg} by
$$\widetilde{N}(V(\Z)^\pm; Y) := \frac{1}{C_{G_0}}\int_{na\in \mathcal{F}} \sum_{\vec{v}\in B^\pm(n,a; Y)\cap (V(\Z)^\pm)^\irr} O(1)^{\omega_{\geq 2}(\Delta(\vec{v}))} \, dn \, d^*a$$
with $Y\asymp X^{12}$, as in \cite[(8) on p.~9]{arulmanjul-5Sel} (but replacing the weight function $1$ by $O(1)^{\omega_{\geq 2}(\Delta(\vec{v}))}$ and noting that their height $H(\vec{v})\gg H(A,B)^{12}$). We now split the integral by the size of the torus parameter $\vec{s}$ that determines $a = a(\vec{s}) \in A'$:
\begin{align}
\widetilde{N}(V(\Z)^\pm; Y)
&= \frac{1}{C_{G_0}}\left(\int_{na\in \mathcal{F} : ||\vec{s}||_\infty\leq Y^\eta} + \int_{na\in \mathcal{F} : ||\vec{s}||_\infty > Y^\eta}\right) \sum_{\vec{v}\in B^\pm(n,a; Y)\cap (V(\Z)^\pm)^\irr} O(1)^{\omega_{\geq 2}(\Delta(\vec{v}))} \, dn \, d^*a
\nonumber \\
&\leq \frac{1}{C_{G_0}}\int_{na\in \mathcal{F} : ||\vec{s}||_\infty\leq Y^\eta} \sum_{\vec{v}\in B^\pm(n,a; Y)\cap V(\Z)^\pm} O(1)^{\omega_{\geq 2}(\Delta(\vec{v}))} \, dn \, d^*a
\nonumber\\
&\qquad\qquad + O\left(Y^{o(1)} \int_{na\in \mathcal{F} : ||\vec{s}||_\infty > Y^\eta} |B^\pm(n,a; Y)\cap (V(\Z)^\pm)^\irr| \, dn \, d^*a\right),
\label{eq:Ntildefirstinequality}
\end{align}
where $\eta\in \R^+$ with $\eta\asymp 1$ a small constant. 

To bound the second summand in \eqref{eq:Ntildefirstinequality}, \cite[Proposition $18$]{arulmanjul-5Sel} gives
\begin{align*}
\int_{na\in \mathcal{F} : ||\vec{s}||_\infty > Y^\eta} \left|\{\vec{v}\in B^\pm(n,a; Y)\cap (V(\Z)^\pm)^\irr : a_{12}(\vec{v}) = 0\}\right| \, dn \, d^*a
&\ll N(V(\Z)^\irr(0); Y) \\
&\ll Y^{{5}/{6} - \Omega(1)}
\end{align*}
where $V(\Z)^\irr(0)$ is the set of strongly irreducible $\vec{v} \in V(\Z)$ where a specific entry $a_{12}$ of $\vec{v}$ vanishes.
But from the proof of \cite[Proposition $18$]{arulmanjul-5Sel}, we have
$$\left| \{\vec{v}\in B^\pm(n,a; Y)\cap (V(\Z)^\pm)^\irr : a_{12}(\vec{v})\neq 0\} \right| = 0$$
if $Y^{1/{60}} w(a_{12})\ll 1$, where $w(a_{12}) = s_1^{-3} s_2^{-6} s_3^{-4} s_4^{-2} s_5^{-4} s_6^{-3} s_7^{-2} s_8^{-1}$ as defined in the proof of \cite[Proposition $18$]{arulmanjul-5Sel}. So
\begin{align}
&\int_{na\in \mathcal{F} : ||\vec{s}||_\infty > Y^\eta} |B^\pm(n,a; Y)\cap (V(\Z)^\pm)^\irr| \, dn \, d^*a
\nonumber\\
&\qquad \ll Y^{{5}/{6} - \Omega(1)} + \int_{na\in \mathcal{F} : ||\vec{s}||_\infty > Y^\eta \text{ and } Y^{1/60} w(a_{12})\gg 1} |B^\pm(n,a; Y)\cap (V(\Z)^\pm)^\irr| \, dn \, d^*a.
\label{eq:bound1}
\end{align}
Now just as in the deduction of $(15)$ from $(14)$ in \cite[p.~16]{arulmanjul-5Sel},  Davenport's Lemma (e.g., \cite[Proposition $17$]{arulmanjul-5Sel}) implies that when $Y^{{1}/{60}} w(a_{12})\gg 1$, we have 
$$\left|B^\pm(n,a; Y)\cap (V(\Z)^\pm)^\irr \right|
\leq \left|B^\pm(n,a; Y)\cap V(\Z)^\pm \right|
\ll \Vol(B^\pm(n,a; Y)).$$
But then
\begin{align}
&\int_{na\in \mathcal{F} : ||\vec{s}||_\infty > Y^\eta \text{ and } Y^{{1}/{60}} w(a_{12})\gg 1} \left|B^\pm(n,a; Y)\cap (V(\Z)^\pm)^\irr\right| \, dn \, d^*a 
\nonumber\\
&\qquad\qquad \ll \int_{na\in \mathcal{F} : ||\vec{s}||_\infty > Y^\eta \text{ and } Y^{{1}/{60}} w(a_{12})\gg 1} \Vol(B^\pm(n,a; Y)) \, dn \, d^*a 
\nonumber\\
&\qquad\qquad \ll \Vol(B^\pm(1,1; Y)) \int_{na\in \mathcal{F} : ||\vec{s}||_\infty > Y^\eta \text{ and } Y^{{1}/{60}} w(a_{12})\gg 1} dn \, d^*a  
\nonumber\\
&\qquad\qquad \ll Y^{{5}/{6} - \Omega(1)}. \label{eq:bound2}
\end{align}
Replacing the second summand in \eqref{eq:Ntildefirstinequality} with the bounds from \eqref{eq:bound1} and \eqref{eq:bound2} yields
\begin{equation} \label{eq:boundintermediatesum}
\widetilde{N}(V(\Z)^\pm; Y)
\leq \frac{1}{C_{G_0}}\int_{na\in \mathcal{F} : ||\vec{s}||_\infty\leq Y^\eta} \sum_{\vec{v}\in B^\pm(n,a; Y)\cap V(\Z)^\pm} O(1)^{\omega_{\geq 2}(\Delta(\vec{v}))} \, dn \, d^*a
+ O\left(Y^{{5}/{6} - \Omega(1)}\right).
\end{equation}

In order to bound the first summand, we wish to apply Lemma \ref{boundedness of averages of ok weights}. To check the nontrivial hypothesis \ref{the probability hypothesis} in Lemma \ref{boundedness of averages of ok weights}, we repeat the arguments at the end of the proof of Theorem \ref{moment theorem for other families}. A similar Hensel lifting argument proves the analogue of \eqref{number of solutions mod prime squares}. We obtain equidistribution in congruence classes as in \eqref{equidistribution in congruence classes} by applying Davenport's Lemma and using  $||\vec{s}||_\infty\leq Y^\eta$ to conclude that for all $m\leq Y^\eta$ and $\vec{v}_0\in V(\Z/m\Z)$, 
$$\left|\{\vec{v}\in B^\pm(n,a;Y)\cap V(\Z) : \vec{v}\equiv \vec{v}_0\pmod*{m}\} \right|
\ll m^{-\dim{V}} \left| B^\pm(n,a;Y)\cap V(\Z) \right|,$$
since both sides are $\asymp \Vol\left(\frac{B^\pm(n,a;Y) - \vec{v}_0}{m}\right) = m^{-\dim V}\Vol\left({B^\pm(n,a;Y)}\right)$.

Thus, we may now apply Lemma \ref{boundedness of averages of ok weights}, which shows that when $||\vec{s}||_\infty\leq Y^\eta$,
\begin{equation} \label{eq:bound sum}
\sum_{\vec{v}\in B^\pm(n,a; Y)\cap V(\Z)^\pm} O(1)^{\omega_{\geq 2}(\Delta(\vec{v}))}
\ll \sum_{\vec{v}\in B^\pm(n,a; Y)\cap V(\Z)^\pm} 1.
\end{equation}
Integrating \eqref{eq:bound sum} over $\mathcal{F}$ shows that the first summand of \eqref{eq:boundintermediatesum} is bounded by $N(V(\Z)^\pm; Y)$ (the unweighted count).\footnote{This application of Lemma \ref{boundedness of averages of ok weights} is where the $\zeta_K(2)$ term appears, but here we are taking $K = \Q$, so we may just use $1$ on the right-hand side of \eqref{eq:bound sum}.} We thus have
$$\widetilde{N}(V(\Z)^\pm; Y)\ll N(V(\Z)^\pm; Y) + O(Y^{{5}/{6} - \Omega(1)}).$$
But by \cite[Theorem $12$]{arulmanjul-5Sel} and \cite[Remark $13$]{arulmanjul-5Sel}, we have $$N(V(\Z)^\pm; Y)\ll Y^{{5}/{6}}\asymp X^{10}\asymp
\sum_{\substack{(A,B)\in \Funiv(\Z) \\ H(A,B)\leq X}} 1,$$ whence
$$\Avg_{\cE \in \FunivX(\Z)}\left(\sum_{1\neq \alpha\in \Sel_5(E/\Q)} O(1)^{\omega_{\geq 2}(\Delta(\alpha))}\right)
\ll 1$$ as desired.
\end{proof}

\subsection{Weighted averages of $2$- and $3$-Selmer elements of elliptic curves in families with marked points}

We now prove the equality cases of Theorem \ref{moment theorem for other families} for $\FF_1$ and $\FF_2$, using nearly identical arguments as in Section \ref{sec:weightuniv}.

\begin{theorem}\label{the two selmer equality case}
Let $(\mathscr{G},\moment)= (\FF_1,\log_2{3})$ or $(\FF_2, 1)$. Let $\FF\subseteq \mathscr{G}$ be a subfamily of positive lower density (ordering by height). Let $S$ be a finite set of places of $\Q$. Then:
$$\Avg_{\FF}( |\cE(\Z[S^{-1}])|^\moment) \ll_{\FF} 1,$$
where the average is taken over all $\cE\in \FF$ ordered by height.
\end{theorem}

Just as in Section \ref{sec:weightuniv}, after repeating the arguments in the beginning of the proof of Theorem \ref{thm:momentsbound}, Theorem \ref{the two selmer equality case} follows from a weighted average bound:

\begin{theorem}\label{weighted two selmer average bound}
Let $(\GG, d) = (\FF_1, 3)$ or $(\FF_2, 2)$. Then $$\Avg_{\GG}\left(O(1)^{\omega_{\geq 2}(\Delta_{\cE})} |\Sel_d(E)|\right)\ll 1,$$ where the average is taken over all $\cE\in \GG$ ordered by height.
\end{theorem}

\begin{proof}
This proof is completely analogous to the proof of Theorem \ref{weighted five selmer average bound}, except with references to \cite{cofreecounting} instead of \cite{arulmanjul-5Sel}. 
For the remainder of the proof, we use the notation of \cite{cofreecounting}. For example, when $\GG = \FF_1$, let the group $G$ be the quotient of $\SL_3^3$ by the stabilizer $\mu_3^2$ and $V$ be the representation $3 \otimes 3 \otimes 3$ of $G$; when $V = \FF_2$, let $G$ be the quotient of $\SL_4^2$ by the analogous $\mu_2^3$ and $V$ be the representation $2 \otimes 2 \otimes 2 \otimes 2$ of $G$ (see Cases 4 and 7 in \cite[Theorem 3.1]{cofreecounting}). Let $n$ be the dimension of $V$ and $k$ be the degree of the discriminant, so $(n,k) = (27, 36)$ and $(16, 24)$ for $\FF_1$ and $\FF_2$, respectively.
For elliptic curves $E$ whose affine Weierstrass models are in $\FF_1$ and $\FF_2$, the notation $S'(E)$ denotes a subgroup of the $d$-Selmer group $\Sel_d(E/\Q)$ that is generated by the marked points; for 100\% of the curves in $\FF_1$ (resp., $\FF_2$), this group has order $3$ (resp., $4$).

The subset of $V(\R)$ with nonzero discriminant is split into $N$ connected components denoted $V^{(i)}$ for $1 \leq i \leq N$. Each contains a fundamental set $R^{(i)}$ defined in \cite[\S 5.1]{cofreecounting}. Let $\mathcal{F}$ be a fundamental domain for the left action of $G(\Z)$ on $G(\R)$ that is Haar-measurable and contained in a standard Siegel set, written as a subset of a product of the same sort of fundamental domain $\mathcal{F}_j$ for $\SL_j$ ($j = 3$ or $2$ for $\FF_1$ and $\FF_2$, respectively). We may explicitly specify
\begin{align*}
\mathcal{F}_2 &= \{\nu \alpha \kappa : \nu(x) \in N'(\alpha), \alpha(s) \in A', \kappa \in K \} \\
\mathcal{F}_3 &= \{\nu \alpha \kappa : \nu(x,x',x'') \in N'(\alpha), \alpha(t,u) \in A', \kappa \in K \}
\end{align*}
where $A'$ is a certain subset of the appropriate tori (indexed by $s$ or $(t,u)$, resp.), $N'(\alpha)$ is a subset of lower triangular matrices, and $K = \SO_j(\R)$, as in \cite[\S 5.2]{cofreecounting}.
As before, let $G_0 \subset G(\R)$ be a compact semialgebraic left $K$-invariant set that is the closure of an open nonempty set. Let $E^{(i)}(\nu,\alpha,Y)$ denote the multiset $\nu \alpha G_0 \cdot R^{(i)} \cap \{\vec{v} \in V^{(i)} : H(\vec{v}) < X\}$.

We now begin the proof. Because $\Avg_{\cE\in \GG : H(\cE)\leq X}\left(O(1)^{\omega_{\geq 2}(\Delta_{\cE})}\right)\ll 1$ (as shown in the proof of Theorem \ref{moment theorem for other families} using Lemma \ref{boundedness of averages of ok weights}), we reduce to showing
\begin{equation}
\label{eq:takeoutS'}
\Avg_{\cE\in \GG : H(\cE)\leq X}\left(\sum_{\vec{v}\in \Sel_d(E) - S'(E)} O(1)^{\omega_{\geq 2}(\Delta(\vec{v}))}\right)\ll 1.
\end{equation}
We are removing the orbits corresponding to the subgroup $S'(E)$ in $\Sel_d(E)$ because we only wish to modify the count for {\em irreducible} orbits.
We bound the number of irreducible $G(\Q)$-orbits by the number of irreducible $G(\Z)$-orbits (here, we drop the constants $C_{G_0}^{(i)}\asymp 1$):
\begin{align}
&\Avg_{\cE\in \GG : H(\cE)\leq X}\left(\sum_{\vec{v}\in \Sel_d(E) - S'(E)} O(1)^{\omega_{\geq 2}(\Delta(\vec{v}))}\right)
\nonumber \\
& \qquad \qquad \ll \sum_{i = 1}^N \int_{\nu\alpha\in N(\alpha)A'} \sum_{\vec{v}\in E^{(i)}(\nu, \alpha, Y)\cap V(\Z)^\irr} O(1)^{\omega_{\geq 2}(\Delta(\vec{v}))} \, dg
\label{eq:familiesZorbits}
\end{align}
with $Y\asymp X$, as in the derivation of \cite[(22) in \S 7.1]{cofreecounting} (but replacing the weight function $1$ by $O(1)^{\omega_{\geq 2}(\Delta(\vec{v}))}$ and noting that we are using the same height).
Let $\widetilde{N}(V(\Z)^{(i)}; Y)$ denote the $i$th summand in \eqref{eq:familiesZorbits}.
Writing $\vec{\sigma} := \vec{s}$ or $(\vec{t}, \vec{u})$ (where, e.g., $s_j$ corresponds to the torus parameter $s$ in the $j$th factor of $\mathcal{F}_2$ in $\mathcal{F} = \mathcal{F}_2^4$), we split the integral to obtain
\begin{align}
\widetilde{N}(V(\Z)^{(i)}; Y)
&= \left(\int_{\nu\alpha\in N(\alpha)A' : ||\vec{\sigma}||_\infty\leq Y^\eta} + \int_{\nu\alpha\in N(\alpha)A' : ||\vec{\sigma}||_\infty > Y^\eta}\right) \sum_{\vec{v}\in E^{(i)}(\nu, \alpha, Y)\cap V(\Z)^\irr} O(1)^{\omega_{\geq 2}(\Delta(\vec{v}))} \, dg
\nonumber\\
&\leq \int_{\nu\alpha\in N(\alpha)A' : ||\vec{\sigma}||_\infty\leq Y^\eta} \sum_{\vec{v}\in E^{(i)}(\nu, \alpha, Y)\cap V(\Z)} O(1)^{\omega_{\geq 2}(\Delta(\vec{v}))} \, dg
\nonumber\\
&\quad\quad + O\left(Y^{o(1)} \int_{na\in N(\alpha)A' : ||\vec{\sigma}||_\infty > Y^\eta} \left|E^{(i)}(\nu, \alpha, Y)\cap V(\Z)^\irr\right| \, dg\right),
\label{eq:boundNtilde1family}
\end{align}
where $\eta\in \R^+$ with $\eta\asymp 1$ is a small constant.

We use \cite[Proposition $7.3$]{cofreecounting} to see that
\begin{align*}
\int_{\nu\alpha\in \mathcal{F} : ||\vec{\sigma}||_\infty > Y^\eta} \left|\{\vec{v}\in E^{(i)}(\nu, \alpha, Y)\cap V(\Z)^\irr : b_{\mathrm{min}}(\vec{v}) = 0\} \right| \, dg
&\ll Y^{{n}/{k} - \Omega(1)}
\end{align*}
where $b_\mathrm{min}$ is a specific (the $111$ or $1111$) entry of $\vec{v} \in V$.
Again, as seen in $(31)$ of \cite{cofreecounting}, 
$$\left|\{\vec{v}\in E^{(i)}(\nu, \alpha, Y)\cap V(\Z)^\irr : b_{\mathrm{min}}(\vec{v})\neq 0\}\right| = 0$$ 
when $Y^{{1}/{k}}w(b_{\mathrm{min}})\ll 1$, where $w(b_{\mathrm{min}}) = \prod_{j=1}^3 t_j^{-2} u_j^{-1}$ or $\prod_{j=1}^4 s_j^{-1}$ for $\GG = \FF_1$ or $\FF_2$, respectively. Thus, we may write
\begin{align*}
&\int_{\nu\alpha\in N(\alpha) A' : ||\vec{\sigma}||_\infty > Y^\eta} \left|E^{(i)}(\nu, \alpha, Y)\cap V(\Z)^\irr \right| \, dg
\\
&\qquad\qquad \ll Y^{{n}/{k} - \Omega(1)} + \int_{\nu\alpha\in N(\alpha) A' : ||\vec{\sigma}||_\infty > Y^\eta, Y^{{1}/{k}} w(b_{\mathrm{min}})\gg 1} \left|E^{(i)}(\nu, \alpha, Y)\cap V(\Z)^\irr \right| \, dg.
\end{align*}\noindent
By Davenport's Lemma (e.g., \cite[Lemma $7.2$]{cofreecounting}), when $Y^{{1}/{k}}w(b_{\mathrm{min}})\gg 1$, $$\left|E^{(i)}(\nu, \alpha, Y)\cap V(\Z)^\irr\right|
\leq \left| E^{(i)}(\nu, \alpha, Y)\cap V(\Z)^\irr \right|
\ll \Vol(E^{(i)}(\nu, \alpha, Y)),$$
so a computation like \eqref{eq:bound2} gives
$$\int_{\nu\alpha\in N(\alpha) A' : ||\vec{\sigma}||_\infty > Y^\eta, Y^{{1}/{k}} w(b_{\mathrm{min}})\gg 1} \left|E^{(i)}(\nu, \alpha, Y)\cap V(\Z)^\irr \right| \, dg\ll Y^{{n}/{k} - \Omega(\eta)}.$$
Combining this with \eqref{eq:boundNtilde1family} yields
$$\widetilde{N}(V(\Z)^{(i)}; Y) \leq \int_{\nu\alpha\in N(\alpha) A' : ||\vec{\sigma}||_\infty\leq Y^\eta} \sum_{\vec{v}\in E^{(i)}(\nu, \alpha, Y)\cap V(\Z)} O(1)^{\omega_{\geq 2}(\Delta(\vec{v}))} \, dg + O\left(Y^{{n}/{k} - \Omega(1)}\right).$$

Finally, we conclude by applying Lemma \ref{boundedness of averages of ok weights}, again verifying the hypotheses by repeating the Hensel lifting argument proving the analogue of \eqref{number of solutions mod prime squares} and modifying the equidistribution argument proving \eqref{equidistribution in congruence classes} by applying Davenport's Lemma and using the fact that $||\vec{\sigma}||_\infty\leq Y^\eta$. By Lemma \ref{boundedness of averages of ok weights}, when $||\vec{\sigma}||_\infty\leq Y^\eta$, 
$$\sum_{\vec{v}\in E^{(i)}(\nu, \alpha, Y)\cap V(\Z)} O(1)^{\omega_{\geq 2}(\Delta(\vec{v}))}
\ll \sum_{\vec{v}\in E^{(i)}(\nu, \alpha, Y)\cap V(\Z)} 1,$$
which gives
$$\widetilde{N}(V(\Z)^{(i)}; Y)\ll N(V(\Z)^{(i)}; Y) + O(Y^{{n}/{k} - \Omega(1)}),$$
where $N(V(\Z)^{(i)}; Y)$ denotes the number of integral points in $V^{(i)}$ of height less than $Y$.
But by \cite[Theorem $7.1$]{cofreecounting} and \cite[Theorem $9.1$]{cofreecounting},
$$N(V(\Z)^{(i)}; Y)\ll Y^{{n}/{k}}\asymp X^{{n}/{k}}\asymp \sum_{\cE\in \GG : H(\cE)\leq X} 1,$$ so 
we have bounded each summand of \eqref{eq:familiesZorbits} and thus obtain the desired bound \eqref{eq:takeoutS'}.
\end{proof}

\vspace{2\baselineskip}

\bibliography{intpts}

\newcommand{\etalchar}[1]{$^{#1}$}
\def\cfgrv#1{\ifmmode\setbox7\hbox{$\accent"5E#1$}\else
  \setbox7\hbox{\accent"5E#1}\penalty 10000\relax\fi\raise 1\ht7
  \hbox{\lower1.05ex\hbox to 1\wd7{\hss\accent"12\hss}}\penalty 10000
  \hskip-1\wd7\penalty 10000\box7}
\providecommand{\bysame}{\leavevmode\hbox to3em{\hrulefill}\thinspace}
\providecommand{\MR}{\relax\ifhmode\unskip\space\fi MR }
\providecommand{\MRhref}[2]{%
  \href{http://www.ams.org/mathscinet-getitem?mr=#1}{#2}
}
\providecommand{\href}[2]{#2}
\begin{thebibliography}{BST{\etalchar{+}}20}

\bibitem[Abr97]{abramovich}
Dan Abramovich, \emph{Uniformity of stably integral points on elliptic curves},
  Invent. Math. \textbf{127} (1997), no.~2, 307--317. \MR{1427620}

\bibitem[Alp14]{levent-intpts}
Levent Alp\"{o}ge, \emph{The average number of integral points on elliptic
  curves is bounded}, 2014, \url{https://arxiv.org/abs/1412.1047}.

\bibitem[BH16]{coregular}
Manjul Bhargava and Wei Ho, \emph{Coregular spaces and genus one curves},
  Cambridge J. Math. \textbf{4} (2016), no.~1, 1--119.

\bibitem[BH22]{cofreecounting}
\bysame, \emph{On average sizes of {S}elmer groups and ranks in families of
  elliptic curves having marked points}, 2022,
  \url{https://arxiv.org/abs/2207.03309}.

\bibitem[BS87]{bombierischmidt}
E.~Bombieri and W.~M. Schmidt, \emph{On {T}hue's equation}, Invent. Math.
  \textbf{88} (1987), no.~1, 69--81.

\bibitem[BS13a]{arulmanjul-4Sel}
Manjul Bhargava and Arul Shankar, \emph{The average number of elements in the
  4-{S}elmer groups of elliptic curves is 7}, 2013,
  \url{http://arxiv.org/abs/1312.7333}.

\bibitem[BS13b]{arulmanjul-5Sel}
\bysame, \emph{The average size of the 5-{S}elmer group of elliptic curves is
  6, and the average rank is less than 1}, 2013,
  \url{http://arxiv.org/abs/1312.7859}.

\bibitem[BS15a]{arulmanjul-bqcount}
\bysame, \emph{Binary quartic forms having bounded invariants, and the
  boundedness of the average rank of elliptic curves}, Ann. of Math. (2)
  \textbf{181} (2015), no.~1, 191--242.

\bibitem[BS15b]{arulmanjul-tccount}
\bysame, \emph{Ternary cubic forms having bounded invariants, and the existence
  of a positive proportion of elliptic curves having rank 0}, Ann. of Math. (2)
  \textbf{181} (2015), no.~2, 587--621.

\bibitem[BSD63]{bsd}
B.~J. Birch and H.~P.~F. Swinnerton-Dyer, \emph{Notes on elliptic curves. {I}},
  J. Reine Angew. Math. \textbf{212} (1963), 7--25.

\bibitem[BST{\etalchar{+}}20]{BSTTTZ}
M.~Bhargava, A.~Shankar, T.~Taniguchi, F.~Thorne, J.~Tsimerman, and Y.~Zhao,
  \emph{Bounds on 2-torsion in class groups of number fields and integral
  points on elliptic curves}, J. Amer. Math. Soc. \textbf{33} (2020), no.~4,
  1087--1099. \MR{4155220}

\bibitem[BSW22]{BSW-globalfields2}
Manjul Bhargava, Arul Shankar, and Xiaoheng Wang, \emph{Geometry-of-numbers
  methods over global fields {II}: {C}oregular vector spaces}, 2022, preprint.

\bibitem[CFS10]{cremonafisherstoll}
John~E. Cremona, Tom~A. Fisher, and Michael Stoll, \emph{Minimisation and
  reduction of 2-, 3- and 4-coverings of elliptic curves}, Algebra Number
  Theory \textbf{4} (2010), no.~6, 763--820.

\bibitem[Dav92]{sinnoudavid}
Sinnou David, \emph{Autour d'une conjecture de {S}. {L}ang}, Approximations
  diophantiennes et nombres transcendants ({L}uminy, 1990), de Gruyter, Berlin,
  1992, pp.~65--98. \MR{1176524}

\bibitem[Eve84]{evertse-on-equations-in-s-units-and-the-thue-mahler-equation}
J.-H. Evertse, \emph{On equations in {$S$}-units and the {T}hue-{M}ahler
  equation}, Invent. Math. \textbf{75} (1984), no.~3, 561--584. \MR{735341}

\bibitem[Eve97]{evertse}
Jan-Hendrik Evertse, \emph{The number of solutions of the {T}hue-{M}ahler
  equation}, J. Reine Angew. Math. \textbf{482} (1997), 121--149. \MR{1427659}

\bibitem[FS16]{fishersadek-5sel}
Tom Fisher and Mohammad Sadek, \emph{On genus one curves of degree 5 with
  square-free discriminant}, J. Ramanujan Math. Soc. \textbf{31} (2016), no.~4,
  359--383.

\bibitem[HV06]{helfgottvenkatesh}
H.~A. Helfgott and A.~Venkatesh, \emph{Integral points on elliptic curves and
  3-torsion in class groups}, J. Amer. Math. Soc. \textbf{19} (2006), no.~3,
  527--550.

\bibitem[Kim18]{dohyeoungkim-intpts}
Dohyeong Kim, \emph{Descent for the punctured universal elliptic curve, and the
  average number of integral points on elliptic curves}, Acta Arith.
  \textbf{183} (2018), no.~3, 201--222. \MR{3813218}

\bibitem[Mor69]{mordelldiophantine}
L.~J. Mordell, \emph{Diophantine equations}, Pure and Applied Mathematics, Vol.
  30, Academic Press, London-New York, 1969. \MR{0249355}

\bibitem[Sil87]{silverman-quantSiegel}
Joseph~H. Silverman, \emph{A quantitative version of {S}iegel's theorem:
  integral points on elliptic curves and {C}atalan curves}, J. Reine Angew.
  Math. \textbf{378} (1987), 60--100. \MR{895285}

\end{thebibliography}
\bibliographystyle{amsalpha2}

\end{document}